\documentclass[12pt, a4paper]{article}

\usepackage{stmaryrd}
\usepackage{enumerate}
\usepackage{hyperref}
\usepackage{color}
\usepackage{lineno}
\usepackage{graphicx}
\usepackage{ae}
\usepackage{amsmath}
\usepackage{amssymb}
\usepackage{latexsym}
\usepackage{url}
\usepackage{epsfig}
\usepackage{mathrsfs}
\usepackage{amsfonts}
\usepackage{amsthm}
\usepackage{stmaryrd}
\usepackage{cite}
\usepackage{BOONDOX-cal}
\usepackage{indentfirst}
\usepackage{amssymb}
\usepackage{indentfirst}
\numberwithin{equation}{section}
\allowdisplaybreaks  

\def\qed{\hfill$\Box$\vspace{12pt}}

\usepackage{color}

\newtheorem{theorem}{Theorem}[section]

\newtheorem{lemma}[theorem]{Lemma}
\newtheorem{proposition}[theorem]{Proposition}

\newtheorem{definition}[theorem]{Definition}
\newtheorem{hypothesis}[theorem]{Hypothesis}
\newtheorem{remark}[theorem]{Remark}

\title{Existence  of weak solutions to
 stochastic heat equations driven 
by truncated  $\alpha$-stable white noises with
non-Lipschitz  coefficients}
\author{Yongjin Wang$^{1,2}$, Chengxin Yan$^{1,}$\thanks{Corresponding author.}~~and~Xiaowen Zhou$^3$\\
{\small $^1$School of Mathematical Sciences, Nankai University, Tianjin {\rm 300071}, China}\\
{\small $^2$School of Business, Nankai University, Tianjin {\rm 300071}, China}
\\
{\small $^3$Department of Mathematics and Statistics, Concordia University, Montreal {\rm H3G 1M8}, Canada}
\\
{Email: {yjwang@nankai.edu.cn};
~{cxyan@mail.nankai.edu.cn}};
\\
~{xiaowen.zhou@concordia.ca}}
\date{}

\begin{document}

\maketitle

\begin{abstract}
We consider a class of stochastic heat 
equations driven by truncated $\alpha$-stable  
white noises for $\alpha\in(1,2)$ with noise coefficients that are 
continuous but not  necessarily Lipschitz continuous.
We prove the existence of weak solution in probabilistic sense, taking values in two different forms under different conditions,  to such an equation using a weak 
convergence argument on solutions to the approximating stochastic heat equations. 
More precisely,
for $\alpha\in(1,2)$  there exists  a measure-valued  weak solution.
However, 
for $\alpha\in(1,5/3)$ there exists
a function-valued weak solution,
 and in this case
we further show that for $p\in(\alpha,5/3)$
the uniform $p$-th moment in $L^p$-norm of the weak solution is finite,
and that the weak solution is uniformly stochastic
continuous in $L^p$ sense. 
\bigskip


\bigskip

\noindent{{\bf MSC Classification (2020):} Primary: 60H15;
Secondary: 60F05, 60G17}
\end{abstract}

\section{Introduction}
\label{sec1}
In this paper we study the existence of weak solution in  probabilistic sense to the
following non-linear stochastic heat equation
\begin{equation}
\label{eq:originalequation1}
\left\{\begin{array}{lcl}
\dfrac{\partial u(t,x)}{\partial t}=\dfrac{1}{2}
\dfrac{\partial^2u(t,x)}{\partial x^2}+
\varphi(u(t-,x))\dot{L}_{\alpha}(t,x),
&& (t,x)\in (0,\infty)
\times(0,L),\\[0.3cm]
u(0,x)=u_0(x),&&x\in[0,L],\\[0.3cm]
u(t,0)=u(t,L)=0,&& t\in[0,\infty),
\end{array}\right.
\end{equation}
where $L$ is an arbitrary positive constants,
$\dot{L}_{\alpha}$ 
denotes a truncated $\alpha$-stable space-time white 
noise on $[0,\infty)\times[0,L]$ with $\alpha\in(1,2)$, 
the noise coefficient
$\varphi:\mathbb{R}\rightarrow \mathbb{R}$ 
satisfies the hypothesis given below
and the initial function 
$u_0$ is random and measurable.

Before studying the equation 
of  particular form (\ref{eq:originalequation1}), we first consider 
an SPDE
\begin{equation}
\label{GeneralSPDE}
\dfrac{\partial u(t,x)}{\partial t}=\dfrac{1}{2}\dfrac{\partial^2u(t,x)}{\partial
x^2}+G(u(t,x))+H(u(t,x))\dot{F}(t,x),
\quad t\geq0, x\in\mathbb{R},
\end{equation}
in which $G:\mathbb{R}\rightarrow \mathbb{R}$
is Lipschitz continuous, $H:\mathbb{R}\rightarrow 
\mathbb{R}$
is continuous and $\dot{F}$ is a 
space-time white noise.

When $\dot{F}$ is  a Gaussian 
white noise, there is a growing literature on SPDEs related to (\ref{GeneralSPDE})
such as the stochastic 
heat (parabolic) equations (see, e.g., Walsh \cite{Walsh:1986} and Dalang et al. \cite{Dalang:2009}),
stochastic Burgers type equations (see, e.g., Bertini and Cancrini \cite{Bertini:1994}, Da Prato et al. \cite{Daprato:1994}), 
SPDEs with reflection
(see, e.g., Zhang \cite{Zhang:2016}),
Parabolic Anderson Model (see, e.g., G\"{a}rtner and Molchanov \cite{Gartner:1990}), 
dissipative stochastic systems or
reaction-diffusion stochastic equations
(see, e.g.,
Cerrai \cite{Cerrai:2003} and
Bignamini \cite{Bignamini:2023} with Wiener noise, Marinelli and R\"{o}ckner \cite{Marinelli:2010} and 
Debussche et al. \cite{Debussche:2013} with L\'{e}vy noise),
SPDEs in Hilbert space (see, e.g., 
Da Prato and Zabczyk \cite{Prato:2014},
Liu and 
R{\"o}ckner \cite{Liu:2010,Liu:2013,Liu:2015}
and references therein), etc.
In particular, such a SPDE arises from super-processes (see, e.g.,  Konno and Shiga \cite{Konno:1988}, Dawson \cite{Dawson:1993}, Perkins \cite{P1991} 
and Ren et al. \cite{Ren:2017,Ren:2018}
and references therein). For $G\equiv 0$ and
$H(u)=\sqrt{u}$,  the solution to (\ref{GeneralSPDE}) is the density field of a one-dimensional super-Brownian motion. 
For $H(u)=\sqrt{u(1-u)}$ (stepping-stone model in population genetics), 
Bo and Wang \cite{Bo:2011}
considered a stochastic interacting model
consisting of equations (\ref{GeneralSPDE})
and proved the existence of weak solution 
in probabilistic sense
to the system by using a weak convergence argument.

In the case that
$\dot{F}$ 
is a Gaussian colored  noise that is white in time and colored in space, for continuous function 
$H$ satisfying the linear growth condition,
Sturm \cite{Sturm:2003} proved the existence of a pair $(u,F)$ satisfying
(\ref{GeneralSPDE}), the so-called weak 
solution in probabilistic sense, by first establishing the existence and uniqueness of lattice systems of SDEs driven by correlated Brownian motions with non-Lipschitz diffusion coefficients that describe branching particle systems in random environment in which the motion process has a discrete Laplacian generator and the branching mechanism is affected by a colored Gaussian random field,
and then applying an approximation procedure. 
Xiong and Yang \cite{Xiong:2023} 
proved the existence of weak solution 
$(u,F)$ to (\ref{GeneralSPDE}) in a finite spatial domain with different boundary conditions
by considering the weak 
limit of a sequence of approximating SPDEs of (\ref{GeneralSPDE}). They further
proved the existence and  uniqueness of the strong solution under additional H\"{o}lder continuity assumption on $H$.

If $\dot{F}$ is a  L\'{e}vy noise 
with Lipschitz continuous coefficient $H$,
there is also a growing literature on existence
and uniqueness of solutions to SPDEs related
to (\ref{GeneralSPDE}).
See, e.g., Albeverio et al. \cite{Albeverio:1998}, Hausenblas \cite{Hausenblas:2005}  and 
R\"{o}ckner and Zhang \cite{Rockner:2007}
for Poisson white noise;
Applebaum and Wu \cite{Applebaum:2000} for a L\'{e}vy space-time 
white noise that extended
the results in \cite{Albeverio:1998};
Peszat and Zabczyk \cite{Peszat:2007}
for infinite dimensional L\'{e}vy 
processes;
Truman and Wu \cite{Truman:2003},
Dong and Xu \cite{Dong:2007},
Wu and Xie \cite{Wu:2012} and Hausenblas and Giri \cite{Hausenblas:2013} for stochastic Burgers type equations;
Bo and Wang \cite{Bo:2006}
for Stochastic {C}ahn-{H}illiard partial differential equations;
Brze\'{z}niak et al. \cite{ Brzezniak:2013},
Dong et al. \cite{Dong:2011} and Zhu et al \cite{Zhu:2019} for
stochastic Navier-Stokes equations; 
Dunst et al. \cite{Dunst:2012}
for numerical approximations.
Brze{\'z}niak et al. \cite{Brzezniak:2014}
studied general SPDEs driven by 
L\'{e}vy processes in Hilbert space (excluding the L\'{e}vy space-time white noise) with 
locally monotone coefficients, and extended
the results in \cite{Liu:2010}.
Brze{\'z}niak et al. \cite{Brzezniak:2018}
studied the existence of weak solution to stochastic reaction-diffusion equations driven by real-valued and general 
Banach-space-valued L\'{e}vy processes.

In particular, when $\dot{F}$ is 
an $\alpha$-stable white noise for $\alpha\in (0,1)\cup(1,2)$,
Balan \cite{Balan:2014} studied SPDE
(\ref{GeneralSPDE}) with $G\equiv 0$ and Lipschitz coefficient $H$ on a bounded domain in $\mathbb{R}^d$ with
zero initial condition and Dirichlet
boundary, and proved the  existence 
of strong random field solution. 
The approach in \cite{Balan:2014}  is 
to first solve the equation with truncated noise
(by removing the big jumps, the jumps size exceeds 
a fixed value $K$, from $\dot{F}$), 
yielding a solution $u_K$,
and then show that for $N\geq K$
the solutions $u_N=u_K$ 
on the event $t\leq\tau_K$, where 
$\{\tau_K\}_{K\geq1}$ is a sequence of stopping
times which tends to infinity
as $K$ tends to infinity. 
Such a  localization method
which is also applied in 
Peszat and Zabczyk 
\cite{Peszat:2006} to show the existence of weak Hilbert-space-valued solution. 
For the stochastic heat equations driven by additive $l^2$-valued $\alpha$-stable processes, we refer to Priola and Zabczyk \cite{Priola:2011} and references therein.

For $\alpha\in(1,2)$, Wang et al. \cite{Wang:2023}  studied the existence and pathwise uniqueness of strong function-valued solution of (\ref{GeneralSPDE}) with Lipschitz coefficient $H$ using a localization method, and  showed a comparison principle of solutions to such equation with different initial functions and drift coefficients.
Yang and Zhou \cite{Yang:2017} found sufficient conditions on
pathwise uniqueness of solutions to a class of SPDEs (\ref{GeneralSPDE}) driven by $\alpha$-stable white noise without negative jumps and with non-decreasing H\"{o}lder continuous noise coefficient $H$. But the existence of weak solution to (\ref{GeneralSPDE}) with general non-decreasing H\"{o}lder continuous noise coefficient is left open.
For stochastic heat equations driven by general heavy-tailed
noises with Lipschitz noise coefficients, we refer to Chong \cite{Chong:2017} and references therein.

When $G=0$, $H(u)=u^{\beta}$ with $0<\beta<1$
(non-Lipschitz continuous) in (\ref{GeneralSPDE}) and 
$\dot{F}$ is an $\alpha$-stable ($\alpha\in(1,2)$) white noise on 
$[0,\infty)\times\mathbb{R}$ without negative
jumps, it is shown in
Mytnik \cite{Mytnik:2002} 
that there exists a weak solution $(u,F)$ satisfying
(\ref{GeneralSPDE}) by constructing a sequence of approximating processes that is tight with its limit solving the associated martingale problem,
and that in the case of $\alpha\beta=1$
the weak uniqueness 
of solution to (\ref{GeneralSPDE}) holds.
The martingale problem approach in
\cite{Mytnik:2002} depends primarily on the 
Laplace transform of $\alpha$-stable noise and
super-process theory which requires 
the assumption of non-negativity of jumps.
The pathwise uniqueness  is shown in \cite{Yang:2017}
for $\alpha\beta=1$ and $1<\alpha<\sqrt{5}-1$.

For $\alpha$-stable colored  noise $\dot{F}$ without negative jumps and  with H\"{o}lder continuous  coefficient $H$, Xiong and Yang \cite{Xiong:2019} proved the existence of 
weak
solution $(u,F)$ to (\ref{GeneralSPDE}) by showing the weak convergence of solutions to    SDE systems on rescaled lattice with discrete Laplacian and driven by common stable random measure, which is similar to \cite{Sturm:2003}.
In both \cite{Sturm:2003} and \cite{Xiong:2019} the dependence of colored  noise helps with establishing the existence of weak solution.

Inspired by work in the above mentioned literature, 
we are interested in the stochastic heat equation
(\ref{eq:originalequation1}) in which 
the noise coefficient $\varphi:\mathbb{R}\rightarrow\mathbb{R}$
satisfies the following more general hypothesis:
\begin{hypothesis}
\label{Hypo} 
$\varphi:\mathbb{R}\rightarrow \mathbb{R}$
is a continuous function with globally linear growth.
\end{hypothesis}

We consider two types of weak solutions in probabilistic sense that are 
measure-valued and function-valued, respectively.
In addition, 
we also study the uniform $p$-moment and
uniform stochastic continuity 
 of the weak solution to
equation (\ref{eq:originalequation1}).
In the case that
$\varphi$ is Lipschitz continuous, the existence
of the strong solution can be usually obtained by standard 
Picard iteration \cite{Walsh:1986,Dalang:2009}
or  Banach fixed point principle
\cite{Truman:2003,Bo:2006,Wang:2023}.
We thus mainly consider the case
that $\varphi$ is non-Lipschitz continuous. 
Since the classical approaches of 
Picard iteration and  Banach fixed point principle
fail for SPDE (\ref{eq:originalequation1}) with non-Lipschitz 
$\varphi$, 
to prove the existence of a weak solution  
$(u,L_{\alpha})$ to
(\ref{eq:originalequation1}), we first construct 
an approximating SPDE sequence with 
Lipschitz continuous noise coefficients 
$\varphi^n$ by using a convolution approximation, and give the existence and 
uniqueness of strong solutions 
to the approximating SPDEs. 
We then proceed to show that 
the sequence of solution is tight
in appropriate spaces. Finally, we prove that
there exists a weak solution of 
(\ref{eq:originalequation1}) by using a weak 
convergence procedure.

The main contribution of this paper
is proving the existence and regularity of  weak solutions to  equation
(\ref{eq:originalequation1}) under general
continuity assumption on noise coefficients,
which generalize  those in many previous works 
\cite{Albeverio:1998,Applebaum:2000,Balan:2014,Bo:2006, Peszat:2006, Dalang:2009, 
Bo:2011, Chong:2017, Wang:2023}
under Lipschitz continuity assumption.
In contrast to the martingale problem approach
used in \cite{Mytnik:2002}, our method aims to directly construct a sequence of approximating SPDEs that
does not rely on the specific form of noise coefficient and
the assumption of non-negativity of jumps, such that we can consider general non-Lipschitz continuous noise coefficients and noise with negative jumps. Our results also generalize the
result in \cite{Mytnik:2002} under truncated
$\alpha$-stable space-time white noise setting.
In addition, we show that there exists 
a function-valued solution of equation
(\ref{eq:originalequation1}) under further assumption $\alpha\in(1,3/5)$. To the best of our knowledge, this result is new for the
stochastic heat equation driven by  
truncated $\alpha$-stable space-time 
white noise with non-Lipschitz noise coefficients.

The rest of this paper is organized as follows. 
In the next section, we introduce some notation 
and the main theorems on the existence, uniform $p$-moment
and uniform stochastic continuity of weak solution to 
(\ref{eq:originalequation1}). 
Section \ref{sec3} is devoted to the proof of the existence of 
measure-valued  weak solution to (\ref{eq:originalequation1}).
In Section \ref{sec4}, for $\alpha\in(1,5/3)$
we prove that there exists a weak solution to (\ref{eq:originalequation1}) as an
$L^p$-valued process
with $p\in(\alpha,5/3)$, 
and that the weak solution 
has the finite uniform $p$-th moment 
and the uniform stochastic continuity in the $L^p$ norm  with $p\in(\alpha,5/3)$.

\section{Notation and main results}
\label{sec2}
\subsection{Notation}
\label{sec2.1}
Let $(\Omega, \mathcal{F}, (\mathcal{F}_t)_{t\geq0}, \mathbb{P})$ be a complete probability 
space with filtration
$(\mathcal{F}_t)_{t\geq0}$ satisfying the usual 
conditions, and let 
$
N(dt,dx,dz): [0,\infty)\times[0,L]\times
\mathbb{R}\setminus\{0\}\rightarrow \mathbb{N}
\cup\{0\}\cup\{\infty\} 
$
be a Poisson random measure on
$(\Omega, \mathcal{F}, (\mathcal{F}_t)_{t\geq0}, 
\mathbb{P})$ with intensity measure 
$dtdx\nu_{\alpha}(dz)$, where $dtdx$ denotes the Lebesgue measure on $[0,\infty)\times[0,L]$ and the
jump size measure $\nu_{\alpha}(dz)$ for $ \alpha\in(1,2) $ is given by
\begin{align}
\label{eq:smalljumpsizemeasure}
\nu_{\alpha}(dz):=(c_{+}z^{-\alpha-1}1_{(0,K]}(z)+c_{-}(-z)^{-\alpha-1}1_{[-K,0)}(z)
)dz,
\end{align}
where $c_{+}+c_{-}=1$ and 
$K>0$ is an arbitrary constant.
Define
\begin{align*}
\tilde{N}(dt,dx,dz):=
N(dt,dx,dz)-dtdx\nu_{\alpha}(dz).
\end{align*}
Then $\tilde{N}(dt,dx,dz)$
is the compensated Poisson random measure
(martingale measure)
on $[0,\infty)\times[0,L]\times
\mathbb{R}\setminus\{0\}$.
As in Balan \cite[Section 5]{Balan:2014}, 
define a martingale measure 
\begin{align}
\label{def:stablenoise}
L_{\alpha}(dt,dx):=
\int_{\mathbb{R}\setminus\{0\}}z\tilde{N}(dt,dx,dz)
\end{align} 
for $(t, x)\in [0,\infty)\times[0,L]$.
Then the corresponding distribution-valued derivative 
$\{\dot{L}_{\alpha}(t,x):t\in[0,\infty),x\in[0,L]\}$
is a truncated $\alpha$-stable space-time
white noise.  
Write $\mathcal{G}^{\alpha}$ for the class of almost
surely $\alpha$-integrable random functions defined by
\begin{align*}
\mathcal{G}^{\alpha}:=\left\{f\in\mathbb{B}:
\int_0^t\int_0^L\vert f(s,x)\vert ^{\alpha}dxds<\infty,
\mathbb{P}\text{-a.s.}\,\,\text{for all} \,\,t\in[0,\infty)\right\},
\end{align*}
where $\mathbb{B}$ is the space of progressively
measurable functions on 
$[0,\infty)\times[0,L]\times\Omega$. Then it
holds by \cite[Section 5]{Mytnik:2002} that the 
stochastic integral with respect to 
$\{L_{\alpha}(dx,ds)\}$
is well defined for all 
$f\in \mathcal{G}^{\alpha}$.

Throughout this paper, $C$ denotes the arbitrary 
positive constant whose value might vary 
from line to line. 
If $C$ depends on some parameters such as $p,T$, 
we denote it by $C_{p,T}$.

Let $G_t(x,y)$ be the fundamental 
solution of heat equation
$\frac{\partial u}{\partial t}
=\frac{1}{2}\frac{\partial^2 u}{\partial x^2}$ 
on the domain $[0,\infty)\times[0,L]\times[0,L]$
with  Dirichlet boundary conditions
(the subscript $t$ is not a derivative but a variable). 
Its explicit formula
(see, e.g., Feller \cite[Page 341]{Feller:1971})
is given by
\begin{equation*}
G_t(x,y)=\dfrac{1}{\sqrt{2\pi t}}\sum_{k=-\infty}^{+\infty}\left\{
\exp\left(-\dfrac{(y-x+2kL)^2}{2t}\right)
-\exp\left(-\dfrac{(y+x+2kL)^2}{2t}
\right)\right\}
\end{equation*}
for $t\in(0,\infty),x,y\in[0,L]$; and $\lim_{t\downarrow0}G_t(x,y)=\delta_y(x)$, 
 where
$\delta$ is the Dirac delta distribution. 
Moreover, it holds 
that for $s,t\in[0,\infty)$ and $x,y,z\in[0,L]$
\begin{equation}
\label{eq:Greenetimation0}
G_t(x,y)=G_t(y,x),\,\,
\int_0^L|G_t(x,y)|dy+\int_0^L|G_t(x,y)|dx\leq C,
\end{equation}
\begin{equation}
\label{eq:Greenetimation1}
\int_0^LG_s(x,y)G_t(y,z)dy=G_{t+s}(x,z),
\end{equation}
\begin{equation}
\label{eq:Greenetimation2}
\int_0^L\vert G_t(x,y)\vert ^pdy\leq Ct^{-\frac{p-1}{2}},\,\, p\geq1.
\end{equation}

Given a topological space $V$, 
let $D([0,\infty),V)$ 
be the space of c\`{a}dl\`{a}g paths from 
$[0,\infty)$
to $V$ equipped with the Skorokhod topology.
For any $p\geq1$,
we denote by $v_t\equiv\{v(t,\cdot),t\in[0,\infty)\}$ 
the $L^p([0,L])$-valued
process equipped with  norm
\begin{equation*}
\vert\vert v_t\vert\vert_{p}=\left(\int_0^L\vert v(t,x)\vert^pdx\right)^{\frac{1}{p}}.
\end{equation*}
For any $p\geq1$ and $T>0$,
let $L_{loc}^p([0,\infty)\times[0,L])$ be the
space of measurable functions $f$ on 
$[0,\infty)\times[0,L])$ such that
\begin{equation*}
\vert \vert f\vert \vert _{p,T}=\left(\int_0^T\int_0^L
\vert f(t,x)\vert ^pdxdt\right)^{\frac{1}{p}}<\infty,\,\,
\forall\,\, 0<T<\infty.
\end{equation*}

Let $\mathcal{S}([0,L])$ be the Schwartz space (the space of rapidly decreasing functions) on $[0,L]$. Let $B([0,L])$ be the space of all Borel functions on $[0,L]$, and
let $\mathbb{M}([0,L])$ be the space of finite 
Borel measures on $[0,L]$ equipped with 
the weak convergence topology.
For any $f\in B([0,L])$ and 
$\mu\in\mathbb{M}([0,L])$ define
$
\langle f,\mu\rangle:=\int_0^L f(x)\mu(dx)
$
whenever it exists.
With a slight abuse of notation, for any $f,g\in B([0,L])$ we also denote by
$
\langle f,g\rangle=\int_0^L f(x)g(x)dx.
$

\begin{remark}
There are two different approaches  of in the study of SPDEs:
the  approach of martingale measure (see, e.g., Walsh \cite{Walsh:1986})
and the approach of  infinite dimensional process
(see, e.g., Da Prato and Zabczyk  \cite{Prato:2014} and Peszat and Zabczyk \cite{Peszat:2007}). 
For a comparison of two approaches 
under the Gaussian setting, 
we refer to Dalang and Quer-Sardanyons \cite{Dalang:2011} and references therein.
In this paper, we adopt Walsh's martingale measure treatment to the L\'{e}vy space-time white noises. The noise can also be represented by the time derivative of impulsive cylindrical L\'{e}vy process (infinite dimensional L\'{e}vy process) in Peszat and Zabczyk 
\cite[Section 7.2]{Peszat:2007}. 
\end{remark}

\subsection{Main results}
\label{sec2.2}
By a solution  to
equation (\ref{eq:originalequation1})
we mean a process $u_t\equiv\{u(t,\cdot),t\in[0,\infty)\}$ satisfying the 
following analytical weak form equation:
\begin{align}
\label{eq:variationform}
\langle u_t,\psi\rangle
&=\langle u_0,\psi\rangle+\dfrac{1}{2}\int_0^t
\langle u_s,\psi{''}\rangle ds
+\int_0^{t+}\int_0^L\int_{\mathbb{R}\setminus\{0\}}
\varphi(u(s-,x))\psi(x)z\tilde{N}(ds,dx,dz)
\end{align}
for all $t\in[0,\infty)$ and for any $\psi\in \mathcal{S}([0,L])$
with $\psi(0)=\psi(L)=\psi^{'}(0)=\psi^{'}(L)=0$ or equivalently
satisfying the
following mild form equation:
\begin{align}
\label{eq:mildform}
u(t,x)=
\int_0^LG_t(x,y)u_0(y)dy
+\int_0^{t+}\int_0^L
\int_{\mathbb{R}\setminus\{0\}}G_{t-s}(x,y)
\varphi(u(s-,y))z\tilde{N}(ds,dy,dz)
\end{align}
for all $t\in [0, \infty)$ and 
for a.e. $x\in [0,L]$, where the last terms in above equations follow from
(\ref{def:stablenoise}). 
For the equivalence between the weak form equation (\ref{eq:variationform}) and mild form equation (\ref{eq:mildform}),
we refer to Yang and Zhou \cite[Proposition 2.2]{Yang:2017} or Li \cite[Theorem 7.26]{Li:2011}
and references therein.
We first give the definition 
of a weak solution in probabilistic sense
to equation (\ref{eq:originalequation1}).
\begin{definition}
Stochastic heat equation (\ref{eq:originalequation1})
has a weak solution in probabilistic sense with initial function $u_0$
if there exists a pair $(u,L_{\alpha})$
defined on some filtered probability space such that 
${L}_{\alpha}$ is a truncated $\alpha$-stable martingale measure on 
$[0,\infty)\times[0,L]$ and
$(u,L_{\alpha})$  satisfies either 
equation (\ref{eq:variationform}) 
or equation (\ref{eq:mildform}).
\end{definition}

We now state the main theorems in this paper. The first theorem is on the existence of weak solution in $D([0,\infty),\mathbb{M}([0,L]))\cap L_{loc}^p([0,\infty)\times [0,L])$ with $p\in(\alpha,2]$
 that was first  considered in  
 Mytnik \cite{Mytnik:2002}. 
\begin{theorem}
\label{th:mainresult} 
If the initial function $u_0$ 
satisfies $\mathbb{E}[\vert\vert u_0\vert\vert_p^p]<\infty$
for some $p\in(\alpha,2]$, then under
 {\rm Hypothesis \ref{Hypo}}
there exists a weak solution 
$(\hat{u}, {\hat{L}}_{\alpha})$ to equation (\ref{eq:originalequation1})
defined on a filtered probability space
$(\hat{\Omega}, \hat{\mathcal{F}}, \{\hat{\mathcal{F}}_t\}_{t\geq0},
\hat{\mathbb{P}})$ such that 
\begin{itemize}
\item[\rm (i)] 
$\hat{u}\in 
D([0,\infty),\mathbb{M}([0,L]))\cap L_{loc}^p([0,\infty)\times [0,L])$;
\item[\rm (ii)] ${\hat{L}}_{\alpha}$ is a truncated 
$\alpha$-stable martingale measure
with the same distribution as ${L}_{\alpha}$.
\end{itemize}
Moreover, for any $T>0$ we have
\begin{equation}
\label{eq:momentresult}
\hat{\mathbb{E}}\left[\vert\vert\hat{u}\vert\vert_{p,T}^p\right]=
\hat{\mathbb{E}}\left[\int_0^T\vert\vert\hat{u}_t\vert\vert_p^pdt\right]<\infty.
\end{equation}
\end{theorem}

The proof of Theorem \ref{th:mainresult} is deferred to Section \ref{sec3}.

Under additional assumption on $\alpha$, 
we can show that there exists a weak solution in 
$D([0,\infty),L^p([0,L]))$,
$p\in(\alpha,5/3)$
 with better regularity.
\begin{theorem}
\label{th:mainresult2}
Suppose that $\alpha\in (1,5/3)$. 
If the initial function $u_0$ 
satisfies $\mathbb{E}[||u_0||_p^p]<\infty$
for some $p\in(\alpha,5/3)$, 
then under {\rm Hypothesis \ref{Hypo}}
there exists a weak solution 
$(\hat{u}, {\hat{L}}_{\alpha})$ to equation (\ref{eq:originalequation1})
defined on a filtered probability space
$(\hat{\Omega}, \hat{\mathcal{F}}, \{\hat{\mathcal{F}}_t\}_{t\geq0},
\hat{\mathbb{P}})$ such that
\begin{itemize}
\item[\rm (i)] 
$\hat{u}\in D([0,\infty),L^p([0,L]))$; 
\item[\rm (ii)] ${\hat{L}}_{\alpha}$ is a
truncated
$\alpha$-stable martingale measure
with the same distribution as ${L}_{\alpha}$.
\end{itemize}
Furthermore, for any $T>0$ we have the following uniform $p$-moment and uniform stochastic continuity, that is,

\begin{equation}
\label{eq:momentresult2}
\hat{\mathbb{E}}\left[\sup_{0\leq t\leq T}\vert\vert\hat{u}_t\vert\vert_p^p\right]<\infty,
\end{equation}
and that
for each $0\leq h\leq\delta$ 
\begin{equation}
\label{eq:timeregular}
\lim_{\delta\rightarrow0}\hat{\mathbb{E}}\left[\sup_{0\leq t\leq T}\sup_{0\leq h\leq \delta}\vert\vert
\hat{u}_{t+h}-\hat{u}_t\vert\vert_p^p\right]=0.
\end{equation}
\end{theorem}

The proof of Theorem \ref{th:mainresult2} is deferred to Section \ref{sec4}. 

\begin{remark}
Regarding the regularity of the weak solution obtained in {\rm Theorem \ref{th:mainresult2}}, we note that it is stochastically continuous or continuous in probability, even though it is not pathwise continuous. 
\end{remark}

Finally, we end up with this section with discussions of the main results in the following remarks.
\begin{remark}
Note that  the globally linear growth of 
$\varphi$
in {\rm Hypothesis \ref{Hypo}} guarantees the global existence 
of weak solutions. One can remove this condition 
if one only needs the existence of a weak solution up to the 
explosion time. 
On the other hand, the uniqueness of the solution to equation (\ref{eq:originalequation1}) is still an open problem because 
$\varphi$ is non-Lipschitz continuous
and the driven noise is white. 
\end{remark}

\begin{remark}
The weak solutions of equation (\ref{eq:originalequation1}) in {\rm Theorems \ref{th:mainresult}} and {\rm\ref{th:mainresult2}} are proved by  showing the tightness of the approximating solution sequence 
$(u^n)_{n\geq1}$ of equation (\ref{eq:approximatingsolution}); see  
{\rm Propositions \ref{th:tightnessresult}} and {\rm\ref{prop:tightnessresult2}}
in {\rm Sections \ref{sec3}} and {\rm \ref{sec4}}, respectively.
To show that the equation (\ref{eq:originalequation1}) has a function-valued
weak solution, it is necessary to restrict 
$\alpha\in(1,5/3)$ due to a technical reason 
that Doob's maximal inequality can not be directly applied to show the uniform $p$-moment estimate of 
$(u^n)_{n\geq1}$ that is key to the proof
of the tightness for $(u^n)_{n\geq1}$. To this end,
we apply the factorization method in {\rm Lemma
\ref{lem:uniformbound}} for transforming the stochastic integral such that the uniform $p$-moment of $(u^n)_{n\geq1}$ can be obtained.
In order to remove this restriction and consider the case of $\alpha\in(1,2)$, we apply
another tightness criteria, i.e., 
{\rm Lemma \ref{lem:tightcriterion0}}, to show the 
tightness of $(u^n)_{n\geq1}$. However, 
the weak solution of equation (\ref{eq:originalequation1}) is a 
measure-valued process in this situation.
We also note that
the existence of function-valued
weak solution of equation (\ref{eq:originalequation1}) in the case of 
$\alpha\in[5/3,2)$ is still an unsolved problem.
\end{remark}

\begin{remark}
If we remove the restriction of the bounded jumps for the $\alpha$-stable white noise 
$\dot{L}_{\alpha}$ in 
equation (\ref{eq:originalequation1}), 
the jump size measure
$\nu_{\alpha}(dz)$ in (\ref{eq:smalljumpsizemeasure})
becomes
\begin{align*}
\nu_{\alpha}(dz)=(c_{+}z^{-\alpha-1}1_{(0,\infty)}(z)+c_{-}(-z)^{-\alpha-1}1_{(-\infty,0)}(z)
)dz
\end{align*}
for $\alpha\in(1,2)$ and $c_{+}+c_{-}=1$.
As in Wang et al. \cite[Lemma 3.1]{Wang:2023}
we can construct a sequence of truncated 
$\alpha$-stable white noise $\dot{L}^K_{\alpha}$
with the jumps size measure given by 
(\ref{eq:smalljumpsizemeasure}) and a sequence
of stopping times $(\tau_K)_{K\geq1}$ such that
\begin{equation}
\label{eq:stoppingtimes1}
\lim_{K\rightarrow+\infty}
\tau_K=\infty,\,\,\mathbb{P}\text{-a.s.}.
\end{equation}
Similar to equation (\ref{eq:originalequation1}),
for given $K\geq1$, 
we can
consider the following non-linear stochastic heat equation
\begin{equation}
\label{eq:truncatedSHE}
\left\{\begin{array}{lcl}
\dfrac{\partial u_K(t,x)}{\partial t}=\dfrac{1}{2}
\dfrac{\partial^2u_K(t,x)}{\partial x^2}+
\varphi(u_K(t-,x))\dot{L}^K_{\alpha}(t,x),
&& (t,x)\in (0,\infty)
\times(0,L),\\[0.3cm]
u_K(0,x)=u_0(x),&&x\in[0,L],\\[0.3cm]
u_K(t,0)=u_K(t,L)=0,&& t\in[0,\infty).
\end{array}\right.
\end{equation}

If $\varphi$ is Lipschitz continuous,
similar to the proof of 
{\rm Proposition \ref{th:Approximainresult}} in 
Wang et al. \cite[Proposition 3.2]{Wang:2023}
one can show that there exists a unique strong solution 
$u_K=\{u_K(t,\cdot),t\in[0,\infty)\}$ to equation
(\ref{eq:truncatedSHE}) by using the
Banach fixed point principle. On the other hand,  
by Wang et al. \cite[Lemma 3.4]{Wang:2023},
it holds for each $K\leq N$ that
\begin{align*}
u_K=u_N
\,\,\mathbb{P}\text{-} \,a.s. 
\,\,\text{on}\,\{t<\tau_K\}.
\end{align*}
By setting 
\begin{align*}
u=u_K,\,0\leq t<\tau_K,
\end{align*}
and by the fact (\ref{eq:stoppingtimes1}),
we obtain the strong (weak) solution 
$u$ to equation
(\ref{eq:originalequation1}) with noise of
unbounded
jumps via letting $K\uparrow+\infty$.

If $\varphi$ is non-Lipschitz continuous,
for any  $K\geq1$,
{\rm Theorem \ref{th:mainresult}} or 
{\rm Theorem \ref{th:mainresult2}}
shows that there exists a weak solution 
$(\hat{u}_K, {\hat{L}}^K_{\alpha})$ to equation (\ref{eq:truncatedSHE})
defined on a filtered probability space
$(\hat{\Omega}, \hat{\mathcal{F}}, \{\hat{\mathcal{F}}_t\}_{t\geq0},
\hat{\mathbb{P}})_K$. 
However, we can not
show that for each $K\leq N$ 
\begin{align*}
(\hat{u}_K, {\hat{L}}^K_{\alpha})=(\hat{u}_N, {\hat{L}}^N_{\alpha})
\,\,\mathbb{P}\text{-} \,a.s. 
\,\,\text{on}\,\,\{t<\tau_K\}
\end{align*}
due to the non-Lipschitz continuity of 
$\varphi$. Therefore, we do not know
whether there exists a common probability
space $(\hat{\Omega}, \hat{\mathcal{F}}, \{\hat{\mathcal{F}}_t\}_{t\geq0},
\hat{\mathbb{P}})$ on which all of the 
weak solutions $((\hat{u}_K, {\hat{L}}^K_{\alpha}))_{K\geq1}$ are defined. 
Hence, the localization
method in  Wang et al. \cite{Wang:2023} 
becomes invalid,
and the existence of the weak solution
to equation (\ref{eq:originalequation1}) with untruncated $\alpha$-stable noise remains an unsolved problem.
\end{remark} 

\begin{remark}
It is not difficult to 
consider the stochastic heat equation 
(\ref{eq:originalequation1}) 
with Dirichlet boundary conditions on a bounded domain, which is given as
follows: 
\begin{equation}
\label{eq:SHEd-dim}
\left\{\begin{array}{lcl}
\dfrac{\partial u(t,x)}{\partial t}=\dfrac{1}{2}
\dfrac{\partial^2u(t,x)}{\partial x^2}+
\varphi(u(t-,x))\dot{L}_{\alpha}(t,x),
&& (t,x)\in (0,\infty)
\times\mathcal{O},\\[0.3cm]
u(0,x)=u_0(x),&&x\in\mathcal{O},\\[0.3cm]
u(t,x)=0,&& (t,x)\in[0,\infty)\times\partial\mathcal{O},
\end{array}\right.
\end{equation}
where $\mathcal{O}$ is a bounded domain of $\mathbb{R}^d(d\geq2)$ with boundary 
$\partial\mathcal{O}$ of class $C^3$. 
The mild form of the above equation is given by
\begin{align*}
u(t,x)=
\int_{\mathcal{O}}\tilde{G}_t(x,y)u_0(y)dy
+\int_0^{t+}\int_{\mathcal{O}}
\int_{\mathbb{R}\setminus\{0\}}\tilde{G}_{t-s}(x,y)
\varphi(u(s-,y))L_{\alpha}(ds,dy)
\end{align*}
for all $t\geq0$  and for a.e. $x\in\mathcal{O}$,
where $\tilde{G}_t(x,y)$ is the heat kernel on 
$(0,\infty)\times\mathcal{O}\times\mathcal{O}$ with
Dirichlet boundary conditions
and $\lim_{t\downarrow0}\tilde{G}_t(x,y)=\delta_x(y)$.
In this situation, we can still  construct a 
sequence of approximating SPDEs with noise 
coefficients $\varphi^n$ given by 
{\rm Lemma \ref{le:approxi-varphi}}
in {\rm Section \ref{sec3}}. 
Under the assumption
$\alpha\in(1,\min\{2,1+2/d\})$ 
and $p\in(\alpha,\min\{2,1+2/d\})$,
one can show that
for each $n\geq1$ there exists a 
pathwise unique strong solution $u^n$ of the
approximating SPDE; see, e.g., 
Balan \cite{Balan:2014} and Wang et al. 
\cite{Wang:2023} for more details. 
With the approximating solution sequence in hand, one
can follow the same  procedure in this paper
to prove that {\rm Theorem \ref{th:mainresult}} holds
for equation (\ref{eq:SHEd-dim}) under the assumption
$\alpha\in(1,\min\{2,1+2/d\})$ 
and $p\in(\alpha,\min\{2,1+2/d\})$, and that 
{\rm Theorem \ref{th:mainresult2}} holds
for equation (\ref{eq:SHEd-dim}) under the assumption
$\alpha\in(1,1+2/(2+d))$ 
and $p\in(\alpha,1+2/(2+d))$.
We omit the  details because 
they are the same as the 
proof in this paper for $d=1$, except
replacing the heat kernel estimates 
(\ref{eq:Greenetimation0})-(\ref{eq:Greenetimation2}) by 
\begin{equation*}
\tilde{G}_t(x,y)=\tilde{G}_t(y,x),\,\,
\int_{\mathcal{O}}|\tilde{G}_t(x,y)|dy+\int_{\mathcal{O}}|\tilde{G}_t(x,y)|dx\leq C,
\end{equation*}
\begin{equation*}
\int_{\mathcal{O}}\tilde{G}_s(x,y)\tilde{G}_t(y,z)dy=\tilde{G}_{t+s}(x,z),
\end{equation*}
and
\begin{align*}
\int_{\mathcal{O}}\vert \tilde{G}_t(x,y)\vert ^pdy\leq Ct^{-\frac{d(p-1)}{2}},\,\, p\geq1,\,\,d\geq2.
\end{align*}
For more general kernel estimates, we refer to
Peszat and Zabczyk
\cite[Theorem 2.6]{Peszat:2007} 
and references therein.
\end{remark}

\section{Proof of Theorem \ref{th:mainresult}}\label{sec3}
The proof of Theorem \ref{th:mainresult}
proceeds in the following three steps.
We first
construct a sequence of the approximating SPDEs with globally Lipschitz continuous noise coefficients 
$(\varphi^n)_{n\geq1}$
using a convolution approximation; see Lemma \ref{le:approxi-varphi},
and show that for each fixed $n\geq1$ there exists
a unique strong
solution $u^n$ in $D([0,\infty),L^p([0,L])$ with $p\in(\alpha,2]$
of the approximating SPDE; see Proposition \ref{th:Approximainresult}.
We then prove that the
approximating solution sequence $(u^n)_{n\geq1}$ is tight in both $D([0,\infty),\mathbb{M}([0,L]))$ and $L_{loc}^p([0,\infty)\times[0,L])$
for all $p\in(\alpha,2]$; 
see Proposition \ref{th:tightnessresult}.
Finally, we proceed to show that there exists a weak solution 
$(\hat{u},\hat{L}_{\alpha})$
to equation (\ref{eq:originalequation1}) 
defined on another probability space
$(\hat{\Omega}, \hat{\mathcal{F}}, 
\{\hat{\mathcal{F}}_t\}_{t\geq0},
\hat{\mathbb{P}})$
by applying a weak convergence argument.

To construct a sequence of the approximating SPDEs with
Lipschitz continuous
noise coefficients, we first use the heat kernel
\begin{align*}
P_t(x)=\frac{1}{\sqrt{2\pi t}}e^{-\frac{x^2}{2t}},\,t>0,x\in\mathbb{R}
\end{align*}
to smooth the continuous function 
$\varphi$
in the following lemma. This technique is also
applied in Xiong and Yang \cite{Xiong:2023}
under the Gaussian colored noise setting.
\begin{lemma}
\label{le:approxi-varphi}
For any $n\geq1$  define
\begin{align*}
\varphi^n(x):=\int_{\mathbb{R}}P_{1/n}(x-y)[(\varphi(y)\wedge n)\vee(-n)]dy, \,\,x\in\mathbb{R}.
\end{align*}
Then
$\lim_{n\rightarrow\infty}\varphi^n(x)
=\varphi(x)$ for all $x\in\mathbb{R}$,
and 
$\varphi^n$
satisfies Lipschitz condition for any fixed $n\geq1$, i.e., 
there exists a constant $C_n$ such that
\begin{align*}
|\varphi^n(x_1)-\varphi^n(x_2)|
\leq C_n|x_1-x_2|,\,\,\forall x_1,x_2\in\mathbb{R}.
\end{align*}
\end{lemma}
\begin{proof}
The first conclusion follows from
the fact
$\lim_{n\rightarrow\infty}P_{1/n}(x)=\delta_x$
and the second one is obtained by the fact
$\varphi^n\in C_{b}^{\infty}(\mathbb{R})$.
\qed
\end{proof}

For each fixed $n\geq1$, 
we construct the approximate SPDE
of the form
\begin{equation}
\label{eq:approximatingsolution}
\left\{\begin{array}{lcl}
\dfrac{\partial u^n(t,x)}{\partial t}=\dfrac{1}{2}\dfrac{\partial^2u^n(t,x)}
{\partial x^2}+\varphi^n(u^n(t-,x))\dot{L}_{\alpha}(t,x),&& (t,x)\in (0,\infty)
\times (0,L),\\[0.3cm]
u^n(0,x)=u_0(x),&&x\in [0,L],
\\[0.3cm]
u^n(t,0)=u^n(t,L)=0,&&t\in[0,\infty),
\end{array}\right.
\end{equation}
where the noise coefficient $\varphi^n$ 
is given by Lemma \ref{le:approxi-varphi}.

Given $n\geq1$, by a solution to
equation (\ref{eq:approximatingsolution})
we mean a process $u^n_t\equiv\{u^n(t,\cdot), t\in[0,\infty)\}$ 
satisfying the following
analytical weak form equation:
\begin{align}
\label{eq:approxivariationform}
\langle u^n_t,\psi\rangle
&=\langle u_0,\psi\rangle+\dfrac{1}{2}\int_0^t
\langle u^n_s,\psi{''}\rangle ds
+\int_0^{t+}\int_0^L\int_{\mathbb{R}\setminus\{0\}}
\varphi^n(u^n(s-,x))\psi(x)z\tilde{N}(ds,dx,dz)
\end{align}
for all $t\in[0,\infty)$ and for any 
$\psi\in \mathcal{S}([0,L])$
with $\psi(0)=\psi(L)=\psi^{'}(0)=\psi^{'}(L)=0$ or equivalently
satisfying the
following mild form equation:
\begin{align}
\label{mildformapproxi0}
u^n(t,x)&=\int_0^LG_t(x,y)u_0(y)dy
+\int_0^{t+}\int_0^L
\int_{\mathbb{R}\setminus\{0\}}G_{t-s}(x,y)
\varphi^n(u^n(s-,y))z\tilde{N}(ds,dy,dz)
\end{align}
for all $t\in [0, \infty)$ and  for a.e. $x\in [0, L]$.

We now present the definition
(see also in Wang et al. \cite{Wang:2023}) 
of a strong solution 
to stochastic heat equation (\ref{eq:approximatingsolution}).
\begin{definition}
Given $p\geq 1$, the stochastic heat equation (\ref{eq:approximatingsolution})
has a strong solution in $D([0,\infty),L^p([0,L]))$ 
 with initial function $u_0$
if for a given truncated $\alpha$-stable 
martingale measure $L_{\alpha}$ there exists 
a process
$u^n_t\equiv\{u^n(t,\cdot),t\in[0,\infty)\}$ in $D([0,\infty),L^p([0,L]))$ such that 
either 
equation (\ref{eq:approxivariationform}) 
or equation (\ref{mildformapproxi0}) holds.
\end{definition}

Note that for each $n\geq1$ the noise coefficient 
$\varphi^n$ is not only Lipschitz continuous but also of globally linear growth. Indeed,
for a given $\epsilon>0$ and $n_0\in\mathbb{N}$ 
large enough, Lemma \ref{le:approxi-varphi} and the 
globally linear growth of $\varphi$ imply that
\begin{equation}
\label{eq:glo-lin-growth}
|\varphi^n(x)|\leq|\varphi^n(x)-\varphi(x)|
+|\varphi(x)|\leq\epsilon+C(1+|x|),\,\, 
\forall n\geq n_0,\, \forall x\in\mathbb{R}.
\end{equation}
Therefore, we can use the classical Banach fixed point 
principle to show the existence and pathwise 
uniqueness of the strong solution to equation (\ref{eq:approximatingsolution}). Since the proof is standard, we just state the main result in the following
proposition. For more details of the proof, we refer to  Wang et al. \cite[Proposition 3.2]{Wang:2023} and references
therein. 
Also note that
the same method was  applied in 
Truman and Wu \cite{Truman:2003}
and in Bo and Wang \cite{Bo:2006}
where the stochastic Burgers equation
and the stochastic {C}ahn-{H}illiard equation
driven by L\'{e}vy space-time white noise were studied,
respectively.

\begin{proposition}
\label{th:Approximainresult}
For a given $n\geq1$, if the initial 
function $u_0$
satisfies $\mathbb{E}[||u_0||_p^p]<\infty$
for some $p\in(\alpha,2]$, 
then 
there exists a pathwise unique strong 
solution 
$u^n_t\equiv\{u^n(t,\cdot),t\in[0,\infty)\}$
to equation (\ref{eq:approximatingsolution}) such that for any $T>0$
\begin{equation}
\label{eq:Approximomentresult}
\sup_{n\geq1}\sup_{0\leq t\leq T}\mathbb{E}\left[\vert \vert u^n_t\vert \vert _p^p\right]<\infty.
\end{equation}
\end{proposition}

\begin{remark}
By {\rm Lemma \ref{le:approxi-varphi}}, (\ref{eq:glo-lin-growth}) and 
estimate (\ref{eq:Approximomentresult}),
the stochastic integral on the right-hand side of
(\ref{mildformapproxi0}) is well defined.
\end{remark}

We are going to prove that the approximating 
solution sequence 
$(u^n)_{n\geq1}$ is tight in both
$D([0,\infty),\mathbb{M}([0,L]))$ and 
$L_{loc}^p([0,\infty)\times[0,L])$
for all $p\in(\alpha,2]$
by using the following tightness criteria; 
see, e.g., Xiong and Yang \cite[Lemma 2.2]{Xiong:2019}.
Note that this tightness criteria can be obtained by
Ethier and Kurtz \cite[Theorems 3.9.1, 3.9.4 and 3.2.2]{Ethier:1986}.
\begin{lemma}
\label{lem:tightcriterion0}
Given  a complete and separable metric space $E$,
let $(X^n=\{X^n(t),t\in[0,\infty)\})_{n\geq1}$ be a sequence of stochastic processes 
with sample paths in $D([0,\infty),E)$, and let $C_a$ be a subalgebra and dense subset of 
$C_b(E)$ (the bounded continuous functions space on $E$). 
Then the sequence $(X^n)_{n\geq1}$ is tight in $D([0,\infty),E)$ 
if both of the following conditions hold:
\begin{itemize}
\item[\rm (i)] For every $\varepsilon>0$ and $T>0$ 
there exists a compact set 
$\Gamma_{\varepsilon,T}\subset E$ such that
\begin{equation}
\label{eq:tightcriterion1}
\inf_{n\geq1}\mathbb{P}[X^n(t)\in\Gamma_{\varepsilon,T}\,\, \text{for all}\,\, t\in[0,T]
]\geq1-\varepsilon.
\end{equation}
\item[\rm (ii)] 
For each $f\in C_a$, there exists a process $g_n\equiv\{g_n(t),t\in[0,\infty)\}$
such that 
\begin{equation*}
f(X^n(t))-\int_0^tg_n(s)ds
\end{equation*}
is an $(\mathcal{F}_t)$-martingale and
\begin{align}
\label{eq:tight-moment}
\sup_{0\leq t\leq T}\mathbb{E}\left[\vert f(X^n(t))\vert +\vert g_n(t)\vert \right]<\infty
\end{align}
and
\begin{align}
\label{eq:tight-moment2}
\sup_{n\geq1}\mathbb{E}\left[\left(\int_0^T\vert g_n(t)\vert ^qdt\right)^{\frac{1}{q}}\right]<\infty
\end{align}
for each $T>0$ and $q>1$.
\end{itemize}
\end{lemma}

Before showing the tightness of solution sequence 
$(u^n)_{n\geq1}$, we first find a uniform moment estimate in the following lemma.
\begin{lemma}
\label{le:uniformbounded}
For each $n\geq1$ let $u^n$ be the strong solution
to equation (\ref{eq:approximatingsolution}) 
given by 
{\rm Proposition \ref{th:Approximainresult}}.
Then for given $T>0$ and 
$\psi\in \mathcal{S}([0,L])$ with $\psi(0)=\psi(L)=\psi^{'}(0)=\psi^{'}(L)=0$, we have for $p\in(\alpha,2]$
that
\begin{align}
\label{eq:p-uniform}
\sup_{n\geq1}\mathbb{E}\left[\sup_{0\leq t\leq T} 
\left\vert \int_{0}^Lu^n(t,x)\psi(x)dx\right\vert ^p\right]<\infty.
\end{align}
\end{lemma}
\begin{proof}
By (\ref{eq:approxivariationform}), it holds that for each $n\geq1$
\begin{align*}
\mathbb{E}\left[\sup_{0\leq t\leq T} 
\left\vert \int_{0}^Lu^n(t,x)\psi(x)dx\right\vert ^p\right]\leq C_p(A_1+A_2+A_3),
\end{align*}
where 
\begin{align*}
A_1&=\mathbb{E}\left[\left\vert \int_0^Lu_0(x)\psi(x)dx\right\vert ^p\right],
\\
A_2&=\mathbb{E}\left[\sup_{0\leq t\leq T} 
\left\vert \int_0^{t}\int_0^Lu^n(s,x)\psi^{''}(x)dxds
\right\vert ^p\right],
\\
A_3&=\mathbb{E}\left[\sup_{0\leq t\leq T} 
\left\vert \int_0^{t+}\int_0^L\int_{\mathbb{R}\setminus\{0\}}\varphi^n(u^n(s-,x))\psi(x)z\tilde{N}(ds,dx,dz)
\right\vert ^p\right].
\end{align*}

For $p\in(\alpha,2]$ we separately estimate $A_1$, $A_2$ and $A_3$ as follows. 
For $A_1$, it holds by H\"{o}lder's inequality that
\begin{align*}
A_1&\leq C_p\left(\int_0^L\vert \psi(x)\vert ^{\frac{p}{p-1}}dx\right)^{\frac{p(p-1)}{p}}
\mathbb{E}\left[\int_0^L\vert u_0(x)\vert ^pdx\right]
\leq C_p\mathbb{E}[\vert \vert u_0\vert \vert _p^p]
\leq C_p
\end{align*}
due to $\psi\in \mathcal{S}([0,L])$ and $\mathbb{E}[\vert \vert u_0\vert \vert _p^p]<\infty$.

For $A_2$, it holds by $\psi\in \mathcal{S}([0,L])$
and H\"{o}lder's inequality that
\begin{align*}
A_2&\leq C_p\mathbb{E}\left[\sup_{0\leq t\leq T} 
\left\vert \int_0^{t}\int_0^Lu^n(s,x)\psi(x)dxds
\right\vert ^p\right]
\leq C_{p,T}\int_0^{T}
\mathbb{E}\left[\sup_{0\leq r\leq s} 
\left\vert \int_0^Lu^n(r,x)\psi(x)dx
\right\vert ^p\right]ds.
\end{align*}

For $A_3$, the Doob maximal inequality and the Burkholder-Davis-Gundy inequality
 imply that
\begin{align*}
A_3&\leq C_p \mathbb{E}\left[\left|\int_0^{T}\int_0^L\int_{\mathbb{R}\setminus\{0\}}\vert \varphi^n(u^n(s-,x))
\psi(x)z\vert^2N(ds,dx,dz)\right|^{\frac{p}{2}}\right]
\\
&\leq C_p \mathbb{E}\left[\int_0^{T}\int_0^L\int_{\mathbb{R}\setminus\{0\}}\vert \varphi^n(u^n(s-,x))
\psi(x)z\vert ^pN(ds,dx,dz)\right]
\\
&=C_p \mathbb{E}\left[\int_0^{T}\int_0^L\int_{\mathbb{R}\setminus\{0\}}\vert \varphi^n(u^n(s,x))
\psi(x)z\vert ^pdsdx\nu_{\alpha}(dz)\right],
\end{align*}
where the second inequality follows from the fact that
\begin{align}
\label{eq:element-inequ}
\left\vert \sum_{i=1}^ka_i^2\right\vert ^{\frac{q}{2}}\leq 
\sum_{i=1}^{k}\vert a_i\vert ^q
\end{align}
for $a_i\in\mathbb{R}, k\geq1$, 
and $q\in(0,2]$.
By (\ref{eq:smalljumpsizemeasure}), it holds that for $p>\alpha$
\begin{align}
\label{eq:jumpestimate}
\int_{\mathbb{R}\setminus\{0\}}\vert z\vert ^p\nu_{\alpha}
(dz)=c_{+}\int_0^Kz^{p-\alpha-1}dz+c_{-}\int_{-K}^0(-z)^{p-\alpha-1}dz
=\frac{K^{p-\alpha}}{p-\alpha},
\end{align}
then there exists a constant $C_{p,K,\alpha}$ such that
\begin{align*}
A_3\leq C_{p,K,\alpha}
\mathbb{E}\left[\int_0^{T}\int_0^L\vert \varphi^n(u^n(s,x))\psi(x)\vert ^pdsdx\right].
\end{align*}

By (\ref{eq:glo-lin-growth}), 
$\psi\in \mathcal{S}([0,L])$
and (\ref{eq:Approximomentresult}) in
Proposition \ref{th:Approximainresult},
it is easy to see that
\begin{align*}
A_3&\leq C_{p,K,\alpha,T}\left(1+\sup_{0\leq s\leq T}
\mathbb{E}[\vert \vert u^n_s\vert \vert _p^p]\right)\leq
C_{p,K,\alpha,T}.
\end{align*}
Combining the estimates $A_1,A_2$ and $A_3$, we have
\begin{align*}
\mathbb{E}&\left[\sup_{0\leq t\leq T} 
\left\vert \int_{0}^Lu^n(t,x)\psi(x)dx\right\vert ^p\right]
\leq 
C_{p,K,\alpha,T}+C_{p,T}\int_0^{T}\mathbb{E}\left[\sup_{0\leq r\leq s} 
\left\vert \int_0^Lu^n(r,x)\psi(x)dx\right\vert ^p\right]ds
\end{align*}
Therefore, it holds by  Gronwall's lemma that for $p\in(\alpha,2]$
\begin{align*}
\sup_{n\geq1}\mathbb{E}\left[\sup_{0\leq t\leq T} 
\left\vert \int_{0}^Lu^n(t,x)\psi(x)dx\right\vert ^p\right]<\infty,
\end{align*}
which completes the proof.
\qed
\end{proof}

Note that for any function 
$v\in L^q[0,L]$ with $q\geq1$
one can identify $L^q([0,L])$ as a subset 
of $\mathbb{M}([0,L])$ (the space of finite Borel measures on $[0,L]$) using the following correspondence
$$v(x)\mapsto v(x)dx.$$
Then for each $n\geq1$ one can identify the 
solution $u^n$ of equation (\ref{eq:approximatingsolution})
as a $\mathbb{M}([0,L])$-valued solution (still denoted by $u^n$). Let $\mathcal{S}([0,L])$ be the Schwartz space (the space of rapidly decreasing functions) on $[0,L]$.
For each $n\geq1$ and 
$\psi\in \mathcal{S}([0,L])$ with 
$\psi(0)=\psi(L)=\psi^{'}(0)=\psi^{'}(L)=0$, let
$\langle u^n,\psi\rangle\equiv\{\langle u_t^n,\psi\rangle,t\geq0\}$
be a real-valued process.
%
Then by (\ref{eq:approxivariationform}) we have
for any $t\geq0$ that
\begin{align*}
\langle u^n_t,\psi\rangle
&=\langle u_0,\psi\rangle+\dfrac{1}{2}\int_0^t
\langle u^n_s,\psi{''}\rangle ds
+\int_0^{t+}\int_0^L\int_{\mathbb{R}\setminus\{0\}}
\varphi^n(u^n(s-,x))\psi(x)z\tilde{N}(ds,dx,dz).
\end{align*}
Since the last term of the above equation is 
a martingale having a  c\`{a}dl\`{a}g 
modification,  it is not difficult to see
that
$(\langle u^n,\psi\rangle)_{n\geq1}$ is a sequence of processes whose sample paths are in $D([0,\infty),\mathbb{R})$, which implies that
$(u^n)_{n\geq1}$ is a sequence of processes whose sample paths are in $D([0,\infty),\mathbb{M}([0,L]))$, see, e.g., Li \cite[Section 12.3]{	Li:2011} and references therein.
We now show the tightness of
$(u^n)_{n\geq1}$ in the following proposition.
\begin{proposition}
\label{th:tightnessresult}
The solution sequence 
$(u^n)_{n\geq1}$ to 
equation (\ref{eq:approximatingsolution}) given by 
{\rm Proposition \ref{th:Approximainresult}}
is tight in both 
$D([0,\infty),\mathbb{M}([0,L]))$ and 
$L_{loc}^p([0,\infty)\times [0,L])$ for $p\in(\alpha,2]$. 
Let $u$ be an arbitrary limit point of $u^n$. 
Then 
\begin{align}
\label{eq:tightresult}
u\in D([0,\infty),\mathbb{M}([0,L]))
\cap L_{loc}^p([0,\infty)\times [0,L])
\end{align}
for $p\in(\alpha,2]$.
\end{proposition}

\begin{proof}
We first prove that the sequence $(\langle u^n,\psi\rangle)_{n\geq1}$
is tight in $D([0,\infty),\mathbb{R})$ by using Lemma  \ref{lem:tightcriterion0}.
It is easy to see that the condition (i) in Lemma \ref{lem:tightcriterion0}
can be verified by Lemma \ref{le:uniformbounded}. 
In the following we mainly
verify the condition (ii) in 
Lemma \ref{lem:tightcriterion0}.

For each $f\in C_b^2(\mathbb{R})$ ($f,f^{'},f^{''}$ are bounded and uniformly continuous)
with compact supports, it holds
by (\ref{eq:approxivariationform}) and 
It\^{o}'s formula (see, e.g., 
\cite[Theorem 4.4.7]{Applebaum:2009}) that
\begin{align}
\label{eq:ito}
f(\langle u_t^n,\psi\rangle)&=f(\langle u_0^n,\psi\rangle)
+\int_0^tf^{'}(\langle u_s^n,\psi\rangle)\langle u_s^n,\psi^{''}\rangle) ds
\nonumber\\
&\quad+\int_0^t\int_0^L\int_{\mathbb{R}\setminus\{0\}}
\mathcal{D}(\langle u_s^n,\psi\rangle,\varphi^n(u^n(s,x))\psi(x)z)
dsdx\nu_{\alpha}(dz)
\nonumber\\
&\quad+\int_0^t\int_0^L\int_{\mathbb{R}\setminus\{0\}}[
f(\langle u_s^n,\psi\rangle+\varphi^n(u^n(s,x))\psi(x)z)-f(\langle u_s^n,\psi\rangle)]\tilde{N}(ds,dx,dz),
\end{align}
where 
$
\mathcal{D}(u,v)=f(u+v)-f(u)-vf^{'}(u)
$
for $u,v\in\mathbb{R}$ and the last term is 
a martingale.

Since $f,f^{'},f^{''}$ are bounded
and $\psi\in \mathcal{S}([0,L])$, then 
\begin{align}
\label{eq:estimate1}
\vert f^{'}(\langle u_s^n,\psi\rangle)\langle u_s^n,\psi^{''}\rangle\vert 
\leq C\left\vert \int_0^Lu^n({s},x)\psi(x)dx\right\vert .
\end{align}
By Taylor's formula, one can show that 
$\vert \mathcal{D}(u,v)\vert \leq C(\vert v\vert\wedge \vert v\vert ^2)$,
which also implies that $\vert \mathcal{D}(u,v)\vert \leq C(\vert v\vert \wedge \vert v\vert ^p)$ 
for $p\in(\alpha,2]$. Thus we have for $p\in(\alpha,2]$,
\begin{align}
\label{eq:estimate2}
&\int_0^L\int_{\mathbb{R}\setminus\{0\}}
\vert \mathcal{D}(\langle u_s^n,\psi\rangle,\varphi^n(u^n(s,x))\psi(x)z)\vert 
dx\nu_{\alpha}(dz)
\nonumber\\
&\leq C\left(\int_{\mathbb{R}\setminus\{0\}}\vert z\vert \wedge \vert z\vert ^p\nu_{\alpha}(dz)\right)
\int_0^L(\vert \varphi^n(u^n(s,x))\psi(x)\vert +\vert \varphi^n(u^n(s,x))\psi(x)\vert ^p)dx
\nonumber\\
&\leq C_{p,K,\alpha}\int_0^L(\vert \varphi^n(u^n(s,x))\psi(x)\vert +\vert \varphi^n(u^n(s,x))\psi(x)\vert ^p)dx,
\end{align}
where by (\ref{eq:smalljumpsizemeasure}),
\begin{align*}
\int_{\mathbb{R}\setminus\{0\}}\vert z\vert \wedge \vert z\vert ^p\nu_{\alpha}(dz)
&=c_{+}\int_0^1z^{p-\alpha-1}dz
+c_{-}\int_{-1}^0(-z)^{p-\alpha-1}dz
+c_{+}\int_1^Kz^{-\alpha}dz
\\
&\quad
+c_{-}\int_{-K}^{-1}(-z)^{-\alpha}dz
\\
&=\frac{1}{p-\alpha}+\dfrac{1-K^{1-\alpha}}{\alpha-1}\leq C_{p,K,\alpha}.
\end{align*}

For given $n\geq1$ let us define
\begin{align*}
g_n(s)&:=f^{'}(\langle u_s^n,\psi\rangle)\langle u_s^n,\psi^{''}\rangle)
+\int_0^L\int_{\mathbb{R}\setminus\{0\}}
\mathcal{D}(\langle u_s^n,\psi\rangle,\varphi^n(u^n(s,x))\psi(x)z)
dx\nu_{\alpha}(dz).
\end{align*}
By (\ref{eq:ito}), it is easy to see that 
\begin{align*}
f(\langle u_t^n,\psi\rangle)-\int_0^{t}g_n(s)ds
\end{align*}
is an $(\mathcal{F}_t)$-martingale. 

Now we verify the moment estimates 
(\ref{eq:tight-moment})
and (\ref{eq:tight-moment2}) of the condition (ii) in 
Lemma \ref{lem:tightcriterion0}.
For each $t\in[0,T]$, it holds by the boundedness of $f$, estimates
 (\ref{eq:estimate1})-(\ref{eq:estimate2}) and (\ref{eq:glo-lin-growth}) that
\begin{align*}
\mathbb{E}\left[\vert f(\langle u_t^n,\psi\rangle)\vert +\vert g_n(t)\vert \right]
&\leq C\left(1+\mathbb{E}\left[\left\vert \int_0^Lu^n(t,x)\psi(x)dx\right\vert \right]\right)
\\
&\quad
+C_{p,K,\alpha}\mathbb{E}\left[\int_0^L(\vert \psi(x)\vert +\vert u^n(t,x)\psi(x)\vert )dx\right]
\\
&\quad+C_{p,K,\alpha}\mathbb{E}\left[\int_0^L(\vert \psi(x)\vert ^p+
\vert u^n(t,x)\psi(x)\vert ^p)dx\right].
\end{align*}
Since $\psi\in \mathcal{S}([0,L])$ implies that $\psi$ is bounded, 
then it holds by H\"{o}lder's inequality and 
(\ref{eq:p-uniform}) that for $p\in(\alpha,2]$
\begin{align*}
\mathbb{E}\left[\vert f(\langle u_t^n,\psi\rangle)\vert +\vert g_n(t)\vert \right]
&\leq C_{p,K,\alpha}\left(1+\left(\sup_{0\leq t\leq T}
\mathbb{E}[\vert \vert u^n_t\vert \vert_p^p]\right)^{\frac{1}{p}}
+\sup_{0\leq t\leq T}
\mathbb{E}[\vert \vert u^n_t\vert \vert_p^p]\right),
\end{align*}
and so by (\ref{eq:Approximomentresult}),
\begin{align*}
\sup_{0\leq t\leq T}\mathbb{E}\left[\vert f(\langle u_t^n,\psi\rangle)\vert +\vert g_n(t)\vert \right]<\infty,
\end{align*}
which verifies the estimate (\ref{eq:tight-moment}).

To verify (\ref{eq:tight-moment2}), it suffices to show that for each $n\geq1$
\begin{align*}
\mathbb{E}\left[\int_0^T\vert g_n(t)\vert ^qdt\right]<\infty
\end{align*}
for some $q>1$.
By the estimates (\ref{eq:estimate1})-(\ref{eq:estimate2}) and (\ref{eq:glo-lin-growth}), we have
\begin{align*}
\mathbb{E}\left[\int_0^T\vert g_n(t)\vert ^qdt\right]
&\leq C_{q}\mathbb{E}\left[\int_0^T\left\vert \int_0^Lu^n(t,x)\psi(x)dx
\right\vert ^qdt\right]
\\
&\quad+C_{p,q,K,\alpha}\mathbb{E}\left[\int_0^T\left\vert \int_0^L(\vert \psi(x)\vert +
\vert u^n(t,x)\psi(x)\vert )dx
\right\vert ^qdt\right]
\\
&\quad+C_{p,q,K,\alpha}\mathbb{E}\left[\int_0^T\left\vert \int_0^L(\vert \psi(x)\vert ^p+
\vert u^n(t,x)\psi(x)\vert ^p)dx
\right\vert ^qdt\right].
\end{align*}
Taking $1<q<2/p$, the H\"{o}lder inequality 
and boundedness of 
$\psi$ imply that
\begin{align*}
\mathbb{E}\left[\int_0^T\vert g_n(t)\vert ^qdt\right]
&\leq C_{p,q,K,\alpha,T}\left(1+\sup_{0\leq t\leq T}\mathbb{E}[\vert \vert u^n_t\vert\vert_q^q]
+\sup_{0\leq t\leq T}\mathbb{E}
\left[\vert \vert u^n_t\vert \vert_{pq}^{pq}\right]\right),
\end{align*}
and so by (\ref{eq:Approximomentresult}),
\begin{align*}
\sup_{n\geq1}\mathbb{E}\left[\int_0^T\vert g_n(t)\vert ^qdt\right]<\infty,
\end{align*}
which verifies the estimate (\ref{eq:tight-moment2}).

Therefore, for each $\psi\in \mathcal{S}([0,L])$ with $\psi(0)=\psi(L)=\psi^{'}(0)=\psi^{'}(L)=0$,
the sequence $(\langle u^n,\psi\rangle)_{n\geq1}$
is tight in $D([0,\infty),\mathbb{R})$, and so 
it holds by Mitoma's theorem 
(see, e.g., Mitoma \cite{Mitoma:1983} and Walsh \cite[pp.361--365]{Walsh:1986}) that
$(u^n)_{n\geq1}$
is tight in $D([0,\infty),\mathbb{M}([0,L]))$.

On the other hand, by (\ref{eq:Approximomentresult}) 
we have for each $T>0$
\begin{align*}
\sup_{n\geq1}\mathbb{E}
\left[\int_0^T\int_0^L\vert u^n(t,x)\vert ^pdxdt\right]
\leq C_T\sup_{n\geq1}\sup_{0\leq t\leq T}
\mathbb{E}[\vert \vert u^n_t\vert \vert _p^p]<\infty
\end{align*}
for $p\in(\alpha,2]$.
The Markov's inequality implies that 
for each $\varepsilon>0,T>0$
there exists a constant $C_{\varepsilon,T}$ such that
\begin{align*}
\sup_{n\geq1}\mathbb{P}\left[\int_0^T\int_0^L\vert u^n(t,x)\vert 
^pdxdt>C_{\epsilon,T}\right]<\varepsilon
\end{align*}
for $p\in(\alpha,2]$.
Therefore, the sequence $(u^n)_{n\geq1}$ is also tight
in $L^p_{loc}([0,\infty)\times[0,L])$ for $p\in(\alpha,2]$, and the conclusion 
(\ref{eq:tightresult}) holds.
\qed
\end{proof}

\begin{Tproof}\textbf{~of Theorem
\ref{th:mainresult}.}
We are going to prove 
Theorem \ref{th:mainresult} by applying  weak 
convergence arguments. For each
$n\geq1$, let $u^n$ be the  
strong solution
 of equation (\ref{eq:approximatingsolution}) given by Proposition \ref{th:Approximainresult}. 
It can also be regarded as an element in 
 $D([0,\infty),\mathbb{M}([0,L]))$
with the Skorokhod topology.
 By Proposition \ref{th:tightnessresult}, there exists 
 a $D([0,\infty),\mathbb{M}([0,L]))\cap L_{loc}^p([0,\infty)\times [0,L])$-valued random variable $u$ such that 
$u^n$ converges to $u$ in distribution in 
 $D([0,\infty),\mathbb{M}([0,L]))\cap L_{loc}^p([0,\infty)\times [0,L])$ for $p\in(\alpha,2]$.
On the other hand, the Skorokhod Representation Theorem
(see, e.g., Either and Kurtz \cite[Theorem 3.1.8]{Ethier:1986}) 
yields that
there exists another filtered probability space
$(\hat{\Omega}, \hat{\mathcal{F}}, (\hat{\mathcal{F}}_t)_{t\geq0},\hat{\mathbb{P}})$ 
and on it a further subsequence 
 $(\hat{u}^n)_{n\geq1}$ and
$\hat{u}$ 
which have the same distribution as 
$(u^n)_{n\geq1}$ and $u$,  
so that
$\hat{u}^n$ almost surely converges to 
$\hat{u}$ in $D([0,\infty),\mathbb{M}([0,L]))\cap L_{loc}^p([0,\infty)\times [0,L])$ for $p\in(\alpha,2]$.

For each $t\geq0,n\geq1$ and any test function 
$\psi\in \mathcal{S}([0,L])$ with 
$\psi(0)=\psi(L)=0$ and
$\psi^{'}(0)=
\psi^{'}(L)=0$, 
let us define 
\begin{align*}
\hat{M}^n_{t}(\psi)&:=\int_0^L\hat{u}^n(t,x)
\psi(x)dx-\int_0^L\hat{u}_0(x)
\psi(x)dx
-\frac{1}{2}\int_0^{t}
\int_0^L\hat{u}^n(s,x)\psi^{''}(x)dxds.
\end{align*}
Since 
$\hat{u}^n$ almost surely converges to 
$\hat{u}$ in the Skorokhod topology
as $n\rightarrow\infty$, then 
\begin{align}
\label{eq:M^n_t}
\hat{M}^n_{t}(\psi)&
\overset{\mathbf{\hat{P}}\text{-a.s.}}{\longrightarrow}
\int_0^L\hat{u}(t,x)\psi(x)dx-
\int_0^L\hat{u}_0(x)\psi(x)dx
-\frac{1}{2}\int_0^{t}
\int_0^L\hat{u}(s,x)
\psi^{''}(x)
dxds
\end{align}
in the Skorokhod topology  as $n\rightarrow\infty$.

By (\ref{eq:approxivariationform}) and the fact that
$\hat{u}^n$ has the same distribution as 
$u^n$ for each $n\geq1$, we have
\begin{align*}
\hat{M}^n_{t}(\psi)
\overset{D}=&
\int_0^Lu^n(t,x)
\psi(x)dx-\int_0^Lu_0(x)\psi(x)dx
-\frac{1}{2}\int_0^{t}\int_0^L
u^n(s,x)\psi^{''}(x)dxds
\nonumber\\
=&\int_0^{t+}
\int_0^L\int_{\mathbb{R}\setminus\{0\}}\psi(x)
\varphi^n(u^n(s-,x))z\tilde{N}(ds,dx,dz),
\end{align*}
where $\overset{D}=$ denotes the identity 
in distribution.  
The Burkholder-Davis-Gundy inequality,
(\ref{eq:element-inequ})-(\ref{eq:jumpestimate}) 
and (\ref{eq:glo-lin-growth}) imply that
for $p\in(\alpha,2]$ 
\begin{align*}
\hat{\mathbb{E}}[\vert \hat{M}^n_{t}(\psi)\vert ^p]
&=\mathbb{E}
\left[\left\vert \int_0^{t+}
\int_0^L\int_{\mathbb{R}\setminus\{0\}}\psi(x)
\varphi^n(u^n(s-,x))z\tilde{N}(ds,dx,dz)
\right\vert ^p\right]
\\
&\leq C_p\mathbb{E}
\left[\int_0^t\int_0^L
\int_{\mathbb{R}\setminus\{0\}}
\vert \psi(x)\vert ^p(1+\vert u^n(s,x)\vert )^p
\vert z\vert ^pdsdx\nu_{\alpha}(dz)
\right]
\\
&\leq C_{p,K,\alpha,T}
\left(\int_0^L\vert \psi(x)\vert ^pdx+\bigg\vert \sup_{x\in[0,L]}\psi(x)\bigg\vert ^p
\sup_{0\leq t\leq T}\mathbb{E}\left[\vert \vert u^n_t\vert \vert _p^p\right]\right).
\end{align*}
Then by $\psi\in \mathcal{S}([0,L])$ and (\ref{eq:Approximomentresult}), 
we have for each $T>0$ 
\begin{align*}
\sup_{n\geq1}\sup_{0\leq t\leq T}
\hat{\mathbb{E}}[\vert \hat{M}^n_{t}(\psi)\vert ^p]<\infty.
\end{align*}
Therefore, it holds by (\ref{eq:M^n_t}) that 
there exists an 
$(\hat{\mathcal{F}}_t)$-martingale 
$\hat{M}_{t}(\psi)$ such that $\hat{M}^n_{t}(\psi)$ 
converges weakly to
$\hat{M}_{t}(\psi)$
as $n\rightarrow\infty$, and for each $t\geq0$
\begin{align}
\label{martingle1}
\hat{M}_{t}(\psi)&=
\int_0^L\hat{u}(t,x)\psi(x)dx-
\int_0^L
\hat{u}_0(x)\psi(x)dx-\frac{1}{2}\int_0^{t}
\int_0^L\hat{u}(s,x)
\psi^{''}(x)dxds.
\end{align}

By Hypothesis \ref{Hypo} and 
Lemma \ref{le:approxi-varphi},
the quadratic variation of
$\{\hat{M}^n_{t}(\psi),t\in[0,\infty)\}$ satisfies that
\begin{align*}
\langle \hat{M}^n(\psi), \hat{M}^n(\psi)
\rangle_t&=\int_0^{t}\int_0^L
\int_{\mathbb{R}\setminus\{0\}}\varphi^n({u}^n(s,x))^2
\psi(x)^2z^2dsdx\nu_{\alpha}(dz)
\\
&\overset{D}=\int_0^{t}\int_0^L
\int_{\mathbb{R}\setminus\{0\}}\varphi^n(\hat{u}^n(s,x))^2
\psi(x)^2z^2dsdx\nu_{\alpha}(dz)
\\
&\overset{\mathbb{P}-a.s.}
\rightarrow\int_0^{t}\int_0^L
\int_{\mathbb{R}\setminus\{0\}}\varphi(\hat{u}(s,x))^2
\psi(x)^2z^2dsdx\nu_{\alpha}(dz),\,\,t\in[0,T],
\end{align*}
 as $n\rightarrow\infty$.
We define by
$\{\langle \hat{M}(\psi),\hat{M}(\psi)\rangle_t,t\in[0,\infty)\}$
the quadratic variation process 
\begin{align*}
\langle \hat{M}(\psi),\hat{M}(\psi)\rangle_t:=\int_0^{t}
\int_0^L\int_{\mathbb{R}\setminus\{0\}}\varphi(\hat{u}(s,x))^2
\psi(x)^2z^2dsdx\nu_{\alpha}(dz),\,\,t\geq0.
\end{align*}
As in 
Konno and Shiga \cite[Lemma 2.4]{Konno:1988}
or Mytnik \cite[Lemma 5.7]{Mytnik:2002},
$\langle \hat{M}(\psi),\hat{M}(\psi)\rangle_t$
corresponds to an orthonormal 
martingale measure $\hat{M}(dt,dx,dz)$ 
defined on the filtered probability space
$(\hat{\Omega}, \hat{\mathcal{F}}, 
(\hat{\mathcal{F}_t})_{t\geq0}, 
\hat{\mathbb{P}})$
in the sense of Walsh
\cite[Chapter 2]{Walsh:1986} whose 
quadratic measure is given by
\begin{align*}
\varphi(\hat{u}(t,x))^2z^2dtdx\nu_{\alpha}(dz).
\end{align*}
Let $\{\dot{\bar{L}}_{\alpha}(t,x):t\in[0,\infty),x\in[0,L]\}$ be another
truncated $\alpha$-stable white noise, 
defined possibly on
$(\hat{\Omega}, \hat{\mathcal{F}}, 
(\hat{\mathcal{F}_t})_{t\geq0}, 
\hat{\mathbb{P}})$,
independent of 
$\hat{M}(dt,dx,dz)$ and define 
\begin{align*}
\hat{L}_{\alpha}(t,\psi)&:=
\int_0^{t+}\int_{0}^L\int_{\mathbb{R}\setminus\{0\}}
\dfrac{1}{\varphi(\hat{u}(s-,x))}
1_{\{\varphi(\hat{u}(s-,x))\neq0\}}\psi(x)z\hat{M}(ds,dx,dz)
\\
&\quad+\int_0^{t+}\int_{0}^L
\psi(x)
1_{\{\varphi(\hat{u}(s-,x))=0\}}
\bar{L}_{\alpha}(ds,dx).
\end{align*}
Then 
$\{\hat{L}_{\alpha}(t,\psi): \,t\in[0,\infty),\, \psi\in \mathcal{S}([0,L]), \psi(0)=\psi(L)=0, \psi^{'}(0)=\psi^{'}(L)=0\}$
determines a truncated
$\alpha$-stable white noise 
$\dot{\hat{L}}_{\alpha}(t,x)$ on 
$(\hat{\Omega}, \hat{\mathcal{F}}, 
(\hat{\mathcal{F}_t})_{t\geq0}, 
\hat{\mathbb{P}})$ with the same distribution as 
$\dot{L}_{\alpha}(t,x)$
such that
\begin{align*}
\hat{M}_t(\psi)&=\int_0^{t+}
\int_0^L\varphi(\hat{u}(s-,x))
\psi(x)
\hat{L}_{\alpha}(ds,dx)
=\int_0^{t+}
\int_0^L\int_{\mathbb{R}\setminus\{0\}}\varphi(\hat{u}(s-,x))
\psi(x)z\widetilde{\hat{N}}(ds,dx,dz),
\end{align*}
where $\widetilde{\hat{N}}(dt,dx,dz)$ denotes
the compensated Poisson random measure
associated to the truncated $\alpha$-stable martingale
measure $\hat{L}_{\alpha}(t,x)$.
Hence, it holds by (\ref{martingle1}) that
$(\hat{u},\hat{L}_{\alpha})$
is a weak solution in probabilistic sense to 
equation
(\ref{eq:originalequation1}) defined on
$(\hat{\Omega}, \hat{\mathcal{F}}, 
(\hat{\mathcal{F}_t})_{t\geq0}, 
\hat{\mathbb{P}})$.

On the other hand,
since $\hat{u}^n$ has the same distribution as 
$u^n$ for each $n\geq1$, then the moment estimates (\ref{eq:Approximomentresult}) in Proposition \ref{th:Approximainresult} can be replaced by 
\begin{equation*}
\sup_{n\geq1}\sup_{0\leq t\leq T}
\hat{\mathbb{E}}\left[\vert \vert \hat{u}^n_t\vert \vert _p^p\right]<\infty
\end{equation*}
for $p\in(\alpha,2]$.
For moment estimate
(\ref{eq:momentresult}), the Fatou's Lemma 
implies that for $p\in(\alpha,2]$
\begin{align*}
\hat{\mathbb{E}}\left[\vert \vert \hat{u}\vert \vert _{p,T}^p\right]&=
\hat{\mathbb{E}}\left[\int_0^T\vert \vert 
\hat{u}_t\vert \vert _p^pdt\right]
\leq\liminf_{n\rightarrow\infty}C_T
\sup_{0\leq t\leq T}\hat{\mathbb{E}}\left[\vert \vert \hat{u}^n_t\vert \vert _p^p\right]<\infty,
\end{align*}
which completes the proof.
\qed
\end{Tproof}

\section{Proof of Theorem \ref{th:mainresult2}}\label{sec4}
The  proof of 
Theorem \ref{th:mainresult2} is
similar to that of Theorem \ref{th:mainresult}.
The main difference between them is that in 
the current proof
we need to prove the solution sequence 
$(u^n)_{n\geq1}$
to equation (\ref{eq:approximatingsolution}), obtained from 
Proposition \ref{th:Approximainresult}, is tight
in $D([0,\infty),L^p([0,L]))$ for $p\in(\alpha,5/3)$.
To this end, 
we need the following tightness criteria; 
see, e.g., Ethier and Kurtz \cite[Theorem 3.8.6 and Remark (a)]
{Ethier:1986}. 
Note that the same criteria was
also applied in Sturm \cite {Sturm:2003} with
Gaussian colored noise setting.

\begin{lemma}
\label{lem:tightcriterion}
Given  a complete and separable metric 
space $(E,\rho)$, 
 let $(X^n)$ be a sequence of stochastic processes 
with sample paths in $D([0,\infty),E)$.
The sequence is tight in $D([0,\infty),E)$ if the 
following conditions hold:
\begin{itemize}
\item[\rm (i)] For every $\varepsilon>0$ and 
rational $t\in[0,T]$, there exists a compact set $
\Gamma_{\varepsilon,T}\subset E$ such that
\begin{equation}
\label{eq:tightcriterion1}
\inf_{n}\mathbb{P}[X^n(t)\in\Gamma_{\varepsilon,T}]
\geq1-\varepsilon.
\end{equation}
\item[\rm (ii)] 
There exists  $p>0$ such that
\begin{equation}
\label{eq:tightcriterion2}
\lim_{\delta\rightarrow0}\sup_n\mathbb{E}
\left[\sup_{0\leq t\leq T}\sup_{0\leq u\leq 
\delta}(\rho(X^n_{t+u},X^n_t)\wedge1)^p\right]=0.
\end{equation}
\end{itemize}
\end{lemma}
To verify condition (i) of Lemma 
\ref{lem:tightcriterion}, we need the following
characterization of the relatively compact set 
in $L^p({[0,L]}),p\geq1$; see, e.g., Sturm \cite[Lemma 4.3]{Sturm:2003}.
\begin{lemma}
\label{lem:compactcriterion}
A subset $\Gamma\subset L^p({[0,L]})$ for  $p\geq1$ 
is relatively compact if and only if 
the following conditions hold:
\begin{itemize}
\item[\rm (a)] $\sup_{f\in\Gamma}
\int_0^L\vert f(x)\vert ^pdx<\infty$,
\item[\rm (b)] $\lim_{y\rightarrow0}\int_0^L\vert 
f(x+y)-f(x)\vert ^pdx=0$ uniformly for all 
$f\in\Gamma$,
\item[\rm (c)] $\lim_{\gamma\rightarrow\infty}
\int_{(L-\frac{L}{\gamma},L]}\vert f(x)\vert ^pdx=0$ 
for all $f\in\Gamma$. 
\end{itemize}
\end{lemma}

The proof of the tightness of $(u^n)_{n\geq1}$
is accomplished
by verifying conditions (i) and (ii) in 
Lemma \ref{lem:tightcriterion}.
To this end, we need some  estimates on
$(u^n)_{n\geq1}$, that is, the uniform bound
estimate in Lemma 
\ref{lem:uniformbound},
the temporal difference estimate in Lemma 
\ref{lem:temporalestimation}
and the spatial difference estimate in Lemma 
\ref{lem:spatialestimation}, respectively.

\begin{lemma}
\label{lem:uniformbound} 
Suppose that $\alpha\in(1,5/3)$ and 
for each $n\geq1$ $u^n$ is the solution
to equation (\ref{eq:approximatingsolution}) 
given by 
{\rm Proposition \ref{th:Approximainresult}}.
Then for given $T>0$ 
there exists a constant 
$C_{p,K,\alpha,T}$ such that
\begin{equation}
\label{eq:unformlybounded}
\sup_n\mathbb{E}\left[\sup_{0\leq t\leq T}\vert \vert u^n_t\vert \vert _p^p\right]\leq C_{p,K,\alpha,T},\,\,\,
\text{for}\,\,\,p\in(\alpha,5/3).
\end{equation}
\end{lemma}

\begin{proof}
For each $n\geq 1$, by (\ref{mildformapproxi0})
it is easy to see that 
$$\mathbb{E}\left[\sup_{0\leq t\leq T}\vert \vert u^n_t\vert \vert _p^p\right]\leq C_p(A_1+A_2),$$
where
\begin{align*}
A_1&=\mathbb{E}\left[\sup_{0\leq t\leq T}\Bigg\vert \Bigg\vert \int_0^LG_{t}(\cdot,y)u_0(y)dy\Bigg\vert \Bigg\vert _p^p\right],\\
A_2&=\mathbb{E}\left[\sup_{0\leq t\leq T}\Bigg\vert \Bigg\vert \int_0^{t+}\int_0^L\int_{\mathbb{R}\setminus\{0\}}
G_{t-s}(\cdot,y)\varphi^n(u^n(s-,y))z
\tilde{N}(ds,dy,dz)\Bigg\vert \Bigg\vert _p^p\right].
\end{align*}
We separately estimate $A_1$ and $A_2$ as follows. 
For $A_1$, it holds by Young's 
convolution inequality 
and (\ref{eq:Greenetimation0}) that
\begin{align*}
\label{eq:A_1}
A_1
&\leq C\mathbb{E}\left[\int_0^L
\sup_{0\leq t\leq T}\left(\int_0^L|G_{t}
(x,y)|dx\right)
\vert u_0(y)\vert ^pdy\right]\leq C_T\mathbb{E}[\vert\vert u_0\vert \vert _p^p].
\end{align*}
By Proposition \ref{th:Approximainresult},
we have $\mathbb{E}[\vert \vert u_0\vert \vert _p^p]<\infty$
for $p\in(\alpha,2]$, and so there exists a 
constant $C_{p,T}$ such that $A_1\leq C_{p,T}$.

For $A_2$, we use the factorization method; 
see, e.g., Da Prato et al. \cite{Prato:1987}, which is based on the fact that for $0<\beta<1$ and $\,0\leq s\leq t$,
\begin{equation*}
\int_s^t(t-r)^{\beta-1}(r-s)^{-\beta}dr=\dfrac{\pi}{\sin(\beta\pi)}.
\end{equation*}
For any function $v: [0,\infty)\times {[0,L]}
\rightarrow \mathbb{R}$ define
\begin{align*}
&\mathcal{J}^{\beta}v(t,x)
:=\dfrac{\sin(\beta\pi)}{\pi}\int_0^t\int_0^L(t-s)^{\beta-1}G_{t-s}(x,y)v(s,y)dyds,\\
&\mathcal{J}^n_{\beta}v(t,x):=\int_0^{t+}\int_0^L\int_{\mathbb{R}\setminus\{0\}}
(t-s)^{-\beta}G_{t-s}(x,y)\varphi^n(v(s-,y))z\tilde{N}(ds,dy,dz).
\end{align*}
By the stochastic Fubini Theorem and (\ref{eq:Greenetimation1}),  we have
\begin{align*}
\mathcal{J}^{\beta}\mathcal{J}^n_{\beta}u^n(t,x)
&=\dfrac{\sin(\beta\pi)}{\pi}\int_0^t\int_0^L
(t-s)^{\beta-1}G_{t-s}(x,y)\Bigg(\int_0^{s+}
\int_0^L\int_{\mathbb{R}\setminus\{0\}}
(s-r)^{-\beta}\\
&\quad\quad\times G_{s-r}(y,m)
\varphi^n(u^n(r-,m))z\tilde{N}(dr,dm,dz)\Bigg)dyds
\\
&=\dfrac{\sin(\beta\pi)}{\pi}\int_0^{t+}
\int_0^L\int_{\mathbb{R}\setminus\{0\}}
\Bigg[\int_r^{t}(t-s)^{\beta-1}(s-r)^{-\beta}\\
&\quad\quad\times\Bigg(\int_0^LG_{t-s}(x,y)
G_{s-r}(y,m)dy\Bigg)ds\Bigg]
\varphi^n(u^n(r-,m))z\tilde{N}(dr,dm,dz)
\\
&=\int_0^{t+}\int_0^L\int_{\mathbb{R}\setminus\{0\}}G_{t-s}(x,y)\varphi^n(u^n(s-,y))z\tilde{N}(ds,dy,dz).
\end{align*}
Thus, $$A_2=\mathbb{E}\left[\sup_{0\leq t\leq T}\vert \vert \mathcal{J}^{\beta}\mathcal{J}^n_{\beta}u^n_t\vert \vert _p^p\right].$$

Until the end of the proof we fix a
$0<\beta<1$ satisfying
\begin{align}
\label{eq:factorization}
1-\frac{1}{p}<\beta<\frac{3}{2p}-\frac{1}{2},
\end{align}
which requires that
$$\frac{3}{2p}-\frac{1}{2}-(1-\frac{1}{p})>0.$$
 Therefore, we need the assumption $p<5/3$ for this lemma.

Back to our  main proof, to estimate $A_2$
we first estimate
$\mathbb{E}[\vert \vert \mathcal{J}^n_{\beta}u^n_t\vert \vert _p^p]$.

For $p\in(\alpha,5/3)$, 
the Burkholder-Davis-Gundy inequality (for  fixed $t$ and varying $s$), (\ref{eq:element-inequ})-(\ref{eq:jumpestimate}) and 
(\ref{eq:glo-lin-growth}) imply that 
\begin{align*}
\mathbb{E}[\vert \vert \mathcal{J}^n_{\beta}u^n_t\vert \vert _p^p]
=&\int_0^L\mathbb{E}\left[\left\vert \int_0^{t+}\int_0^L\int_{\mathbb{R}\setminus\{0\}}
(t-s)^{-\beta}G_{t-s}(x,y)\varphi^n(u^n(s-,y))z\tilde{N}(ds,dy,dz)\right\vert ^p\right]dx\nonumber\\
\leq&C_p\int_0^L\int_0^{t}\int_0^L\int_{\mathbb{R}\setminus\{0\}}
\mathbb{E}\left[\vert (t-s)^{-\beta}G_{t-s}(x,y)\varphi^n(u^n(s,y))z\vert ^p\right]\nu_{\alpha}(dz)dydsdx\nonumber\\
\leq&C_{p,K,\alpha}\int_0^L\int_0^{t}\int_0^L
\mathbb{E}\left[1+\vert u^n(s,y)\vert ^p\right]\vert t-s\vert ^{-\beta p}\vert G_{t-s}(x,y)\vert ^pdydsdx.
\end{align*}

Combine (\ref{eq:Greenetimation2}), we have 
\begin{align*}
\mathbb{E}[\vert \vert \mathcal{J}^n_{\beta}u^n_t\vert \vert _p^p]\leq&C_{p,K,\alpha}\left(L+\mathbb{E}\left[\sup_{0\leq s\leq t}\vert \vert u^n_s\vert \vert _p^p\right]\right)
\int_0^Ts^{-(\frac{p-1}{2}+\beta p)}ds.
\end{align*}
For $p<5/3$, by (\ref{eq:factorization}) we have 
$$\int_0^Ts^{-(\frac{p-1}{2}+\beta p)}ds<\infty.$$ 
Therefore,
there exists a constant $C_{p,K,\alpha,T}$ such that
\begin{equation}
\label{eq:momentestimation}
\mathbb{E}\left[\vert \vert \mathcal{J}^n_{\beta}u^n_t\vert \vert _p^p\right]\leq
C_{p,K,\alpha,T}\left(1+\mathbb{E}\left[\sup_{0\leq s\leq t}\vert \vert u^n_s\vert \vert _p^p\right]\right).
\end{equation}

We now estimate $A_2=\mathbb{E}[\sup_{0\leq t\leq T}\vert \vert \mathcal{J}^{\beta}\mathcal{J}^n_{\beta}u^n_t\vert \vert _p^p]$. The Minkowski inequality implies that
\begin{align}
\label{eq:A_2_1}
A_2
=&\mathbb{E}\left[\sup_{0\leq t\leq T}\dfrac{\sin(\pi\beta)}{\pi}\Bigg\vert \Bigg\vert \int_0^{t}\int_0^L(t-s)^{\beta-1}
G_{t-s}(\cdot,y)\mathcal{J}^n_{\beta}u^n(s,y)dyds \Bigg\vert \Bigg\vert _p^p\right]\nonumber\\
\leq&\dfrac{\sin(\pi\beta)}{\pi}\mathbb{E}\left[\sup_{0\leq t\leq T}\left(\int_0^{t}(t-s)^{\beta-1}
\Bigg\vert \Bigg\vert \int_0^LG_{t-s}(\cdot,y)\mathcal{J}^n_{\beta}u^n(s,y)dy\Bigg\vert \Bigg\vert _pds\right)^p\right].
\end{align}
By the H\"{o}lder inequality and 
(\ref{eq:Greenetimation0}), we have
\begin{align}
\label{eq:A_2_2}
&\Bigg\vert \Bigg\vert \int_0^LG_{t-s}(\cdot-y)\mathcal{J}^n_{\beta}u^n(s,y)dy\Bigg\vert \Bigg\vert _p\nonumber\\
&\quad=\left(\int_0^L\left\vert \int_0^L\vert G_{t-s}(x,y)\vert ^{\frac{p-1}{p}}\vert G_{t-s}(x,y)\vert ^{\frac{1}{p}}\mathcal{J}^n_{\beta}u^n(s,y)
dy\right\vert ^pdx\right)^{\frac{1}{p}}\nonumber\\
&\quad\leq\left(\int_0^L\left\vert \left(\int_0^L\vert G_{t-s}(x,y)\vert dy\right)^{\frac{p-1}{p}}
\left(\int_0^LG_{t-s}(x,y)\vert \mathcal{J}^n_{\beta}u^n(s,y)\vert ^pdy\right)^{\frac{1}{p}}\right\vert ^pdx\right)^{\frac{1}{p}}\nonumber\\
&\quad\leq\left(\sup_{x\in[0,L]}\int_0^L|G_{t-s}(x,y)|dy\right)^{\frac{p-1}{p}}
\left(\int_0^L\int_0^L|G_{t-s}(x,y)|\vert \mathcal{J}^n_{\beta}u^n(s,y)\vert ^pdxdy\right)^{\frac{1}{p}}\nonumber\\
&\quad\leq C_T\vert \vert \mathcal{J}^n_{\beta}u^n_s\vert \vert _p.
\end{align}
Therefore, it follows from (\ref{eq:A_2_1}), (\ref{eq:A_2_2}), and the H\"{o}lder inequality that
\begin{align*}
\label{eq:A_2}
A_2\leq&\dfrac{\sin(\pi\beta)C_{p,T}}{\pi}\mathbb{E}\left[\sup_{0\leq t\leq T}\left(\int_0^t(t-s)^{\beta-1}\vert \vert \mathcal{J}^n_{\beta}u^n_s\vert \vert _pds\right)^p\right]\nonumber\\
\leq&\dfrac{\sin(\pi\beta)C_{p,T}}{\pi}\mathbb{E}\left[\sup_{0\leq t\leq T}\left(\int_0^{t}1^{\frac{p}{p-1}}ds\right)^{p-1}
\left(\int_0^{t}(t-s)^{(\beta-1)p}\vert \vert \mathcal{J}^n_{\beta}u^n_s\vert \vert _p^pds\right)\right]
\nonumber\\
\leq&\dfrac{\sin(\pi\beta)C_{p,T}}{\pi}\int_0^{T}
(T-s)^{(\beta-1)p}
\mathbb{E}[\vert \vert \mathcal{J}^n_{\beta}u^n_s\vert \vert _p^p]ds.
\end{align*}
By (\ref{eq:momentestimation}), 
it also holds that
\begin{align}
A_2
\leq&\dfrac{\sin(\pi\beta)C_{p,K,\alpha,T}}{\pi}
\int_0^{T}(T-s)^{(\beta-1)p}\left(1+\mathbb{E}
\left[\sup_{0\leq r\leq s}\vert \vert u^n_r\vert \vert 
_p^p\right]\right)ds\nonumber\\
\leq&\dfrac{\sin(\pi\beta)C_{p,K,\alpha,T}}{\pi}
\left(1+\int_0^{T}(T-s)^{(\beta-1)p}\mathbb{E}
\left[\sup_{0\leq r\leq s}\vert \vert u^n_r
\vert \vert _p^p\right]ds\right).
\end{align}
Combining (\ref{eq:A_2}) and the estimate for $A_1$, 
we have for each $T>0$,
\begin{align*}
&\mathbb{E}\left[\sup_{0\leq t \leq T}\vert \vert u^n_t\vert \vert _p^p\right]\leq
C_{p,T}
+\dfrac{\sin(\pi\beta)C_{p,K,\alpha,T}}{\pi}
\int_0^{T}(T-s)^{(\beta-1)p}\mathbb{E}
\left[\sup_{0\leq r\leq s}\vert \vert u^n_r\vert \vert 
_p^p\right]ds.
\end{align*}
Since $\beta>1-1/p$,  
applying a generalized Gronwall's Lemma
(see, e.g., Lin \cite[Theorem 1.2]{Lin:2013}), we have 
\begin{align*}
\sup_{n}\mathbb{E}\left[\sup_{0\leq t \leq T}\vert \vert 
u^n_t\vert \vert _p^p\right]\leq C_{p,K,\alpha,T},
\,\,\,\text{for}\,\,\,p\in(\alpha,5/3),
\end{align*}
which completes the proof.
\qed
\end{proof}
\begin{remark}
Note that one can
estimate term $A_2$ in the proof
of Lemma \ref{lem:uniformbound} using the Kotelenets inequality or similar maximal inequalities. For more details, we refer to Marinelli et al. \cite{Marinelli:2009} and references therein.
\end{remark}

\begin{lemma}
\label{lem:temporalestimation}
Suppose that $\alpha\in(1,5/3)$ and 
for each $n\geq1$ $u^n$ is the solution
to equation (\ref{eq:approximatingsolution}) 
given by {\rm Proposition \ref{th:Approximainresult}}. 
Then for given $T>0$, $0\leq h\leq\delta$ 
and $p\in(\alpha,5/3)$
\begin{equation}
\label{eq:temporalestimation}
\lim_{\delta\rightarrow0}\sup_n\mathbb{E}\left[\sup_{0\leq t\leq T}\sup_{0\leq h\leq \delta}\vert \vert 
u^n_{t+h}-u^n_t\vert \vert _p^p\right]=0.
\end{equation}
\end{lemma}

\begin{proof} 
For each $n\geq 1$, by the factorization 
method in the proof of Lemma 
\ref{lem:uniformbound}, we have
$$\mathbb{E}\left[\sup_{0\leq t\leq T}\sup_{0\leq 
h\leq \delta}\vert \vert  u^n_{t+h}-u^n_t\vert \vert 
_p^p\right]\leq
C_p(B_1+B_2),$$
where
\begin{align*}
B_1&=\mathbb{E}\left[\sup_{0\leq t\leq T}
\sup_{0\leq h\leq \delta}\Bigg\vert \Bigg\vert 
\int_0^L(G_{t+h}(\cdot-y)-G_{t}(\cdot-
y))u_0(y)dy\Bigg\vert \Bigg\vert _p^p\right],\\
B_2&=\mathbb{E}\left[\sup_{0\leq t\leq T}
\sup_{0\leq h\leq \delta}\vert \vert \mathcal{J}^{\beta}
\mathcal{J}^n_{\beta}u^n_{t+h}-\mathcal{J}
^{\beta}\mathcal{J}^n_{\beta}u^n_t\vert \vert 
_p^p\right].
\end{align*}
For $B_1$, Young's convolution inequality 
and (\ref{eq:Greenetimation0}) imply that
\begin{align*}
B_1
&\leq \mathbb{E}\left[\int_0^L\sup_{0\leq t\leq T}\sup_{0\leq h\leq \delta}\left(\int_0^L(|G_{t+h}(x,y)|+|G_{t}(x,y)|)dx\right)\vert u_0(y)\vert ^pdy\right]
\leq C_T\mathbb{E}[\vert \vert u_0\vert \vert _p^p]<\infty.
\end{align*}
Therefore, it holds by Lebesgue's 
dominated convergence theorem that
$B_1$ converges to 0 as $\delta\rightarrow0$.

For $B_2$, it is easy to see that 
$$B_2\leq \frac{\sin(\beta\pi)C_p}{\pi}
(B_{2,1}+B_{2,2}+B_{2,3}),$$ 
where
\begin{align*}
B_{2,1}&=\mathbb{E}\Bigg[\sup_{0\leq t\leq T}\sup_{0\leq h\leq \delta}\Bigg\vert \Bigg\vert 
\int_0^{t}\int_0^L(t-s)^{\beta-1}
 (G_{t+h-s}(\cdot,y)-G_{t-s}(\cdot,y))
\mathcal{J}^n_{\beta}u^n(s,y)dyds\Bigg\vert \Bigg\vert _p^p\Bigg],\\
B_{2,2}&=\mathbb{E}\Bigg[\sup_{0\leq t\leq T}\sup_{0\leq h\leq \delta}\Bigg\vert \Bigg\vert 
\int_0^{t}\int_0^L((t+h-s)^{\beta-1}-(t-s)^{\beta-1})G_{t+h-s}(\cdot,y)\mathcal{J}^n_{\beta}u^n(s,y)dyds\Bigg\vert \Bigg\vert _p^p\Bigg],\\
B_{2,3}&=\mathbb{E}\left[\sup_{0\leq t\leq T}\sup_{0\leq h\leq \delta}\Bigg\vert \Bigg\vert 
\int_{t}^{t+h}\int_0^L(t+h-s)^{\beta-1}G_{t+h-s}(\cdot,y)\mathcal{J}^n_{\beta}u^n(s,y)dyds\Bigg\vert \Bigg\vert _p^p\right].
\end{align*}
By the assumption $p\in(\alpha,5/3)$ of this lemma we can choose a $0<\beta<1$ satisfying $1-1/p<\beta<3/2p-1/2$.
By Lemma \ref{lem:uniformbound} and (\ref{eq:momentestimation}), there exists a constant $C_{p,K,\alpha,T}$ such that
\begin{equation}
\label{ineq:0}
\sup_{0\leq t\leq T}\mathbb{E}\left[\vert \vert \mathcal{J}^n_{\beta}u^n_t\vert \vert _p^p\right]\leq C_{p,K,\alpha,T}.
\end{equation}

To estimate $B_{2,1}$, 
we set $G^h_t(x,y)=G_{t+h}(x,y)-G_{t}(x,y)$. 
Similar to the estimates for
(\ref{eq:A_2_1}) and (\ref{eq:A_2_2}) in the proof 
of Lemma \ref{lem:uniformbound}, we have 
\begin{align*}
B_{2,1}\leq&\mathbb{E}\Bigg[\sup_{0\leq t\leq T}
\sup_{0\leq h\leq
\delta}\Bigg(\int_0^{t}(t-s)^{\beta-1}
\Bigg(\sup_{x\in [0,L]}
\int_0^LG^h_{t-s}(x,y)dy\Bigg)^{\frac{p-1}{p}}
\vert \vert \mathcal{J}^n_{\beta}u^n_s\vert \vert 
_pds\Bigg)^p\Bigg].
\end{align*}
It also follows from the H\"{o}lder inequality and (\ref{ineq:0}) that
\begin{align*}
B_{2,1}\leq&C_{p,T}\sup_{0\leq t\leq T}\mathbb{E}\left[\vert \vert \mathcal{J}^n_{\beta}u^n_t\vert \vert _p^p\right]
\sup_{0\leq h\leq \delta}
\left(\int_0^{T}s^{(\beta-1)p}\left(\sup_{x\in [0,L]}\int_0^LG_{t-s}^h(x,y)dy\right)^{p-1}ds\right)\nonumber\\
\leq&C_{p,K,\alpha,T}\sup_{0\leq h\leq \delta}
\left(\int_0^{T}s^{(\beta-1)p}\left(\sup_{x\in [0,L]}\int_0^LG^h_{t-s}(x,y)dy\right)^{p-1}ds\right).
\end{align*}
Moreover, since $\beta>1-1/p$, it holds by (\ref{eq:Greenetimation0}) that
\begin{align*}
&\int_0^{T}s^{(\beta-1)p}\left(\sup_{x\in [0,L]}\int_0^LG_{t-s}^h(x,y)dy\right)^{p-1}ds\\
&\quad\leq\int_0^{T}s^{(\beta-1)p}\left(\sup_{x\in
[0,L]}\left(\int_0^L|G_{t+h-s}(x,y)|dy+\int_0^L|G_{t-s}(x,y)|dy\right)\right)^{p-1}ds\\
&\quad\leq 
C_{p,T}\int_0^{T}s^{(\beta-1)p}ds<\infty.
\end{align*}
Thus,  Lebesgue's Dominated Convergence Theorem implies that 
$B_{2,1}$ converges to 0 as $\delta\rightarrow0$.

For $B_{2,2}$, the Minkowski inequality and Young's convolution inequality imply that
\begin{align*}
B_{2,2}\leq&\mathbb{E}\Bigg[\sup_{0\leq t\leq T}\sup_{0\leq h\leq \delta}
\Bigg(\int_0^{t}((t+h-s)^{\beta-1}-(t-s)^{\beta-1})
\Bigg\vert \Bigg\vert \int_0^LG_{t+h-s}(\cdot,y)\mathcal{J}^n_{\beta}u^n(s,y)dy\Bigg\vert \Bigg\vert _p ds\Bigg)^p\Bigg]\nonumber\\
\leq&C_{p,T}\mathbb{E}\left[\sup_{0\leq t\leq T}\sup_{0\leq h\leq \delta}\left(\int_0^{t}((t+h-s)^{\beta-1}-(t-s)^{\beta-1})
\vert \vert \mathcal{J}^n_{\beta}u^n_s\vert \vert _p ds\right)^p\right].\nonumber
\end{align*}
By the H\"{o}lder inequality and (\ref{ineq:0}) we have for 
$\beta>1-1/p$, 
\begin{align*}
B_{2,2}\leq&C_{p,T}\mathbb{E}\left[\sup_{0\leq t\leq T}\sup_{0\leq h\leq \delta}
\int_0^{t}\vert (t+h-s)^{\beta-1}-(t-s)^{\beta-1}\vert ^{p}
\vert \vert \mathcal{J}^n_{\beta}u^n_s\vert \vert _p^p ds\right]\nonumber\\
\leq&C_{p,T}\sup_{0\leq t\leq T}\mathbb{E}[\vert \vert \mathcal{J}^n_{\beta}u^n_t\vert \vert _p^p]\int_0^{T}\vert (s+\delta)^{\beta-1}-s^{(\beta-1)}\vert ^pds\nonumber\\
\leq&C_{p,K,\alpha,T}\int_0^{T}\vert (s+\delta)^{\beta-1}-s^{(\beta-1)}\vert ^pds<\infty.
\end{align*}
Therefore, by Lebesgue's dominated convergence theorem,
we know that $B_{2,2}$ converges to 0 as $\delta\rightarrow0$.

For $B_{2,3}$, similar to $B_{2,2}$, we get
\begin{align}
\label{ineq:B_{2_3}}
B_{2,3}\leq&\mathbb{E}\left[\sup_{0\leq t\leq T}\sup_{0\leq h\leq \delta}
\left(\int_{t}^{t+h}(t+h-s)^{\beta-1}\Big\vert \Big\vert \int_0^LG_{t+h-s}(\cdot,y)\mathcal{J}^n_{\beta}u^n(s,y)dy\Big\vert \Big\vert _p 
ds\right)^p\right]\nonumber\\
\leq&C_{p,T}\mathbb{E}\left[\sup_{0\leq t\leq T}\sup_{0\leq h\leq \delta}\left(
\int_{t}^{t+h}(t+h-s)^{\beta-1}\vert \vert \mathcal{J}^n_{\beta}u^n_s\vert \vert _p ds\right)^p\right]\nonumber\\
\leq&C_{p,T}\mathbb{E}\left[\sup_{0\leq t\leq T}\sup_{0\leq h\leq \delta}
\int_{t}^{t+h}\vert (t+h-s)^{\beta-1}\vert ^p\vert \vert \mathcal{J}^n_{\beta}u^n_s\vert \vert _p^p ds\right]\nonumber\\
\leq&C_{p,T}\sup_{0\leq t\leq T}\mathbb{E}[\vert \vert \mathcal{J}^n_{\beta}u^n_t\vert \vert _p^p]\sup_{0\leq h\leq \delta}
\int_{t}^{t+h}\vert (t+h-s)^{\beta-1}\vert 
\leq C_{p,K,\alpha,T}\int_0^{\delta}s^{(\beta-1)p}ds.
\end{align}
Since $\beta>1-1/p$, we can conclude that the right-hand side
of (\ref{ineq:B_{2_3}}) converges to 0 as $\delta\rightarrow0$. Therefore, by 
the estimates of $B_{2,1}, B_{2,2}, B_{2,3}$ and $B_1$, the
desired result (\ref{eq:temporalestimation}) holds,
which completes the proof.
\qed
\end{proof}

\begin{lemma}
\label{lem:spatialestimation}
For each $n\geq1$ let $u^n$ be the solution
to equation (\ref{eq:approximatingsolution}) 
given by {\rm Proposition \ref{th:Approximainresult}}.
Then for given $t\in[0,\infty)$, 
$0\leq\vert x_1\vert \leq\delta$ 
and $p\in(\alpha,2]$ 
\begin{equation}
\label{eq:spatialestimation}
\lim_{\delta\rightarrow 0}\sup_n\mathbb{E}\left[ 
\sup_{\vert x_1\vert \leq \delta}\vert \vert u^n(t,\cdot+x_1)-u^n(t,
\cdot)\vert \vert _p^p\right]=0.
\end{equation}
\end{lemma}
\begin{proof}
Since the shift operator is continuous in $L^p([0,L])$, then for each $n\geq1$ 
and $\delta>0$
there exists a pathwise 
$x_1^{n,\delta}(t)\in\mathbb{R}$ such that
$\vert x_1^{n,\delta}(t)\vert \leq\delta$ and 
$$\sup_{\vert x_1\vert \leq \delta}\vert \vert 
u^n(t,\cdot+x_1)-u^n(t,\cdot)\vert \vert _p^p=\vert \vert  u^n(t,\cdot+x_1^{n,\delta}(t))-u^n(t,\cdot)\vert \vert _p^p.$$
As before, it is easy to see that 
$$\mathbb{E}[\vert \vert 
u^n(t,\cdot+x_1^{n,\delta}(t))-u^n(t,\cdot)
\vert \vert _p^p]\leq C_p(C_1+C_2),$$ 
where
\begin{align*}
C_1&=\mathbb{E}\left[
\bigg\vert \bigg\vert \int_0^L
(G_{t}(\cdot+x_1^{n,\delta}(t),y)
-G_{t}(\cdot,y))u_0(y)dy\bigg\vert \bigg\vert _p^p
\right],
\\
C_2&=\mathbb{E}\bigg[\bigg\vert \bigg\vert 
\int_0^{t+}\int_0^L\int_{\mathbb{R}\setminus\{0\}}
(G_{t-s}(\cdot+{x_1^{n,\delta}(t)},y)
-G_{t-s}(\cdot,y))
\varphi^n(u^n(s-,y))
z\tilde{N}(dz,dy,ds)\bigg\vert \bigg\vert _p^p\bigg].
\end{align*}
For $C_1$, Young's convolution inequality
and (\ref{eq:Greenetimation0}) imply that
\begin{align*}
C_1&\leq\mathbb{E}\left[\int_0^L
\left(\int_0^L(|G_{t}(x+x_1^{n,
\delta}(t),y)|+|G_{t}(x,y))|dx\right)
\vert u_0(y)\vert ^pdy\right]
\leq C_T\mathbb{E}[\vert \vert u_0\vert \vert _p^p]<\infty.
\end{align*}
Thus, the Lebesgue dominated convergence theorem implies that $C_1$ converges to 0 as $\delta\rightarrow0$.

For $C_2$, it follows from the Burkholder-Davis-Gundy inequality, (\ref{eq:element-inequ})-(\ref{eq:jumpestimate}), (\ref{eq:glo-lin-growth})
and (\ref{eq:Greenetimation2}) that for $p\in(\alpha,2]$
\begin{align*}
C_2
=&\int_0^L\mathbb{E}\bigg[\bigg\vert \int_0^{t+}\int_0^L\int_{\mathbb{R}\setminus\{0\}}
(G_{t-s}(x+{x_1^{n,\delta}(t)},y)-G_{t-s}(x,y))\varphi^n(u^n(s-,y))z
\tilde{N}(ds,dy,dz)\bigg\vert ^p\bigg]dx\nonumber\\
\leq&C_p\int_0^L\int_0^{t}\int_0^L\int_{\mathbb{R}\setminus\{0\}}
\mathbb{E}[\vert (G_{t-s}(x+{x_1^{n,\delta}(t)},y)-G_{t-s}(x,y))\varphi^n(u^n(s,y))z\vert ^p]\nu_{\alpha}(dz)dydsdx\nonumber\\
\leq&C_{p,K,\alpha}\int_0^L\int_0^{t}\int_0^L\mathbb{E}[(1+\vert u^n(s,y)\vert )^p]
\vert (G_{t-s}(x+{x_1^{n,\delta}(t)},y)-G_{t-s}(x,y))\vert ^pdydsdx\nonumber\\
\leq&C_{p,K,\alpha}\left(\int_0^L\int_0^{t}\vert G_{t-s}(x+{x_1^{n,\delta}(t)},y)-G_{t-s}(x,y)\vert ^pdsdx\right)
\left(L+\sup_{0\leq s\leq t}\mathbb{E}\left[\vert \vert u^n_s\vert \vert _p^p\right]\right)
\\
\leq& 
C_{p,K,\alpha}\left(\int_0^L\int_0^{t}(\vert G_{t-s}(x+{x_1^{n,\delta}(t)},y)\vert ^p+\vert G_{t-s}(x,y)\vert ^p)dsdx\right)
\left(L+\sup_{0\leq s\leq t}\mathbb{E}\left[\vert \vert u^n_s\vert \vert _p^p\right]\right)
\\
\leq&
C_{p,K,\alpha}\left(\int_0^t(t-s)^{-\frac{p-1}{2}}ds\right)
\left(L+\sup_{0\leq s\leq t}\mathbb{E}\left[\vert \vert u^n_s\vert \vert _p^p\right]\right)
\end{align*}

Therefore, it holds 
by (\ref{eq:Approximomentresult}) and
Lebesgue's dominated convergence theorem 
that $C_2$ converges to 0 as $\delta\rightarrow0$. 
Hence,  by the 
estimates of $C_1$ and $C_2$, we obtain
\begin{align*}
\lim_{\delta\rightarrow 0}\sup_n\mathbb{E}\bigg[\sup_{\vert x_1\vert \leq
\delta}\vert \vert u^n(t,\cdot+x_1)-u^n(t,\cdot)\vert \vert _p^p\bigg]&=0,
\end{align*}
which completes the proof.
\qed
\end{proof}

\begin{proposition}
\label{prop:tightnessresult2}
Suppose that $\alpha\in(1,5/3)$.
The sequence of solutions  
$(u^n)_{n\geq1}$ to equation 
(\ref{eq:approximatingsolution}) 
given by {\rm Proposition \ref{th:Approximainresult}}
is tight in  $D([0,\infty),L^p([0,L]))$
for $p\in(\alpha,5/3)$.
\end{proposition}
\begin{proof}
From (\ref{eq:Approximomentresult}) and
Markov's inequality, for each $\varepsilon>0$, $p\in(\alpha,2]$ and $T>0$
there exists a $N\in\mathbb{N}$ such that
\begin{align*}
\sup\limits_n\mathbb{P}\left[\vert \vert u^n_t\vert \vert 
_p^p>N\right]\leq\dfrac{\varepsilon}{3},\quad t\in [0,T].
\end{align*}
Let $\Gamma^1_{\varepsilon,T}$ be a closed set
defined by
\begin{align}
\label{eq:Gamma1}
\Gamma^1_{\varepsilon,T}:=\{v_t\in L^p([0,L]): 
\vert \vert v_t\vert \vert _p^p\leq N,t\in[0,T]\}.
\end{align}

By Lemma \ref{lem:spatialestimation}
and Markov's inequality, it holds that for each 
$\varepsilon>0$, $p\in(\alpha,2]$ and $T>0$ 
\begin{equation*}
\lim_{\delta\rightarrow 0}\sup_n\mathbb{P}\left[ 
\sup_{\vert x_1\vert \leq \delta}\vert \vert 
u^n(t,\cdot+x_1)-u^n(t,\cdot)\vert \vert 
_p^p>\varepsilon\right]=0,\quad t\in[0,T].
\end{equation*}
Then for $k\in\mathbb{N}$ 
we can choose a sequence 
$(\delta_k)_{k\geq1}$ with
$\delta_k\rightarrow0$ as $k\rightarrow\infty$
such that
\begin{align*}
\sup\limits_n\mathbb{P}\left[\sup_{\vert x_1\vert \leq 
\delta_k}\vert \vert 
u^n(t,\cdot+x_1)-u^n(t,\cdot)\vert \vert _p^p>\frac{1}{k}
\right]\leq\dfrac{\varepsilon}{3}2^{-k},\quad t\in[0,T].
\end{align*}
Let $\Gamma^2_{\varepsilon,T}$ be a closed set
defined by
\begin{align}
\label{eq:Gamma2}
\Gamma^2_{\varepsilon,T}:=\bigcap_{k=1}^{\infty}
\left\{v_t\in L^p([0,L]): 
\sup_{\vert x_1\vert \leq \delta_k}\vert \vert 
v(t,\cdot+x_1)-v(t,\cdot)\vert \vert _p^p\leq\frac{1}{k},t\in[0,T]
\right\}.
\end{align}

We next prove that for each $\varepsilon>0$ 
and $p\in(\alpha,2]$, 
\begin{equation}
\label{eq:compactset1}
\lim_{\gamma\rightarrow \infty}\sup_n\mathbb{P}\left[\int_{(L-\frac{L}{\gamma},L]}\vert u^n(t,x)\vert ^pdx>\varepsilon\right]=0.
\end{equation}
It is easy to see that
\begin{align*}
\mathbb{E}\bigg[\int_{(L-\frac{L}{\gamma},L]}\vert u^n(t,x)\vert ^pdx\bigg]
&=\mathbb{E}\bigg[\int_0^L
\vert u^n(t,x)\vert ^p1_{(L-\frac{L}{\gamma},L]}(x)dx\bigg]
\leq C_p(D_1+D_2),
\end{align*}
 where
\begin{align*}
D_1&=\int_0^L\mathbb{E}\left[\left\vert \int_0^LG_{t}
(x,y)u_0(y)dy\right\vert ^p\right]
1_{(L-\frac{L}{\gamma},L]}(x)dx,
\\
D_2&=\int_0^L\mathbb{E}\Bigg[\Bigg\vert \int_0^{t+}
\int_0^L\int_{\mathbb{R}\setminus\{0\}}
G_{t-s}(x,y)
\varphi^n(u^n(s-,y))z
\tilde{N}(ds,dy,dz)
\Bigg\vert ^p\Bigg]1_{(L-\frac{L}{\gamma},L]}(x)dx.
\end{align*}
It is easy to prove that $D_1$ converges to 0 as
$\gamma\rightarrow\infty$ by using Young's convolution
inequality and Lebesgue's dominated convergence 
theorem. For $D_2$, it holds by the 
Burkholder-Davis-Gundy inequality,
(\ref{eq:element-inequ})-(\ref{eq:jumpestimate})
and 
(\ref{eq:glo-lin-growth}) that
for $p\in(\alpha,2]$
\begin{align*}
D_2
\leq &C_p\int_0^L\int_0^{t}
\int_0^L\int_{\mathbb{R}\setminus\{0\}}
\mathbb{E}[\vert G_{t-s}(x,y)\varphi^n(u^n(s,y))z\vert ^p]
1_{(L-\frac{L}{\gamma},L]}(x)\nu_{\alpha}(dz)dydsdx
\nonumber\\
\leq&C_{p,K,\alpha}\int_0^L\int_0^{t}
\int_0^L\mathbb{E}(1+\vert u^n(s,y)\vert )^p\vert G_{t-s}(x,y)\vert 
^p1_{(L-\frac{L}{\gamma},L]}(x)dydsdx
\nonumber\\
\leq&C_{p,K,\alpha,T}\left(L+\sup_{0\leq t\leq T}\mathbb{E}
\left[\vert \vert u^n_t\vert \vert 
_p^p\right]\right)
\int_0^L\int_0^{t}(t-s)^{-\frac{p-1}{2}}
1_{(L-\frac{L}{\gamma},L]}(x)dsdx.
\end{align*}
Since $p\leq2$,
it holds that
\begin{align*}
D_2\leq C_{p,K,\alpha,T}
\left(\int_0^L1_{(L-\frac{L}{\gamma},L]}(x)dx\right)\left(L+\sup_{0\leq t\leq T}\mathbb{E}
\left[\vert \vert u^n_t\vert \vert _p^p\right]\right).
\end{align*}
By (\ref{eq:Approximomentresult}), $D_2$ 
converges to 0 as
$\gamma\rightarrow\infty$.  
Therefore, 
(\ref{eq:compactset1}) is obtained from the 
estimates of $D_1$ and $D_2$ and Markov's 
inequality.

For any $k\in\mathbb{N}$ and $T>0$ we can choose a sequence
$(\gamma_k)_{k\geq1}$ with $\gamma_k\rightarrow\infty$ as 
$k\rightarrow\infty$
such that
\begin{align*}
\sup\limits_n\mathbb{P}\left
[\int_{(L-\frac{L}{\gamma_k},L]}\vert u^n(t,x)\vert ^pdx>\frac{1}{k}\right]
\leq\dfrac{\varepsilon}{3}2^{-k},\quad t\in[0,T].
\end{align*}
Let $\Gamma^3_{\varepsilon,T}$ be a closed set
defined by
\begin{align}
\label{eq:Gamma3}
\Gamma^3_{\varepsilon,T}:=\bigcap_{k=1}^{\infty}
\left\{v_t\in L^p([0,L]): \int_{(L-\frac{L}{\gamma_k},L]}\vert v(t,x)\vert ^pdx\leq\frac{1}{k}, t\in[0,T]
\right\}.
\end{align}

Combining (\ref{eq:Gamma1}), (\ref{eq:Gamma2})
and (\ref{eq:Gamma3}) to define
\begin{align*}
\Gamma_{\varepsilon,T}:=\Gamma^1_{\varepsilon,T}
\cap\Gamma^2_{\varepsilon,T}
\cap\Gamma^3_{\varepsilon,T},
\end{align*}
 then $\Gamma_{\varepsilon,T}$ is a closed set
in $L^p([0,L]),p\in(\alpha,2]$. For any function $f\in\Gamma_{\varepsilon,T}$
the definition of $\Gamma_{\varepsilon,T}$
implies that the conditions (a)-(c) in
Lemma \ref{lem:compactcriterion}
hold, and so $\Gamma_{\varepsilon,T}$ is a 
relatively compact set in 
$L^p([0,L]),p\in(\alpha,2]$. 
Combing the closeness and relatively
compactness, we know that 
$\Gamma_{\varepsilon,T}$ is a 
compact set in $L^p([0,L]),p\in(\alpha,2]$. 
Moreover, the definition of 
$\Gamma_{\varepsilon,T}$
implies that
\begin{align*}
\inf_n\mathbb{P}
[u^n_t\in\Gamma_{\varepsilon,T}]&\geq1-
\frac{\varepsilon}{3}\left(1+2\sum_{k=1}^{\infty}
2^{-k}\right)=1-\varepsilon,
\end{align*}
which verifies condition (i) of Lemma 
\ref{lem:tightcriterion}. Condition (ii) of 
Lemma \ref{lem:tightcriterion} is verified by
Lemma \ref{lem:temporalestimation} with $p\in(\alpha,5/3)$. 
Therefore, $(u^n)_{n\geq1}$ is tight in 
$D([0,\infty),L^p([0,L]))$ for $p\in(\alpha,5/3)$,
which completes the proof.
\qed
\end{proof}

\begin{Tproof}\textbf{~of Theorem
\ref{th:mainresult2}.}
According to Proposition \ref{prop:tightnessresult2},
there exists a $D([0,\infty),L^p([0,T]))$-valued random variable 
$u$ such that $u^n$ converges to $u$ in distribution in
the Skorohod topology.
The Skorohod Representation Theorem
yields that
there exists another filtered probability space
$(\hat{\Omega}, \hat{\mathcal{F}}, (\hat{\mathcal{F}}_t)_{t\geq0},\hat{\mathbb{P}})$ 
and on it a further subsequence 
 $(\hat{u}^n)_{n\geq1}$ and
$\hat{u}$ 
which have the same distribution as 
$(u^n)_{n\geq1}$ and $u$,  
so that
$\hat{u}^n$ almost surely converges to 
$\hat{u}$ in the Skorohod topology. 
The rest of the proofs,
including the construction
of a truncated $\alpha$-stable measure $\hat{L}_{\alpha}$ such that 
$(\hat{u},\hat{L}_{\alpha})$ is a weak solution to equation
(\ref{eq:originalequation1}), 
is same as the proof of Theorem \ref{th:mainresult} and 
we omit them.

Since $\hat{u}^n$ has the same distribution as 
$u^n$ for each $n\geq1$,
the moment estimate (\ref{eq:unformlybounded}) in Lemma \ref{lem:uniformbound} can be 
written as
$$
\sup_{n\geq1}\hat{\mathbb{E}}\left[\sup_{0\leq t\leq T}\vert \vert \hat{u}^n_t\vert \vert _p^p\right]\leq C_{p,K,\alpha,T}.
$$
Hence, by Fatou's Lemma,
$$
\hat{\mathbb{E}}\left[\sup_{0\leq t\leq T}\vert \vert \hat{u}_t\vert \vert _p^p\right]\leq\liminf_{n\rightarrow\infty}
\hat{\mathbb{E}}\left[\sup_{0\leq t\leq T}\vert \vert \hat{u}^n_t\vert \vert _p^p\right]<\infty.
$$
This yields the uniform $p$-moment estimate
(\ref{eq:momentresult2}). Similarly, we can obtain the
uniform stochastic
continuity (\ref{eq:timeregular}) by  
Lemma \ref{lem:temporalestimation}.
\qed
\end{Tproof}

\noindent \textbf{Acknowledgements}
The authors would like to thank editors and anonymous referees for careful reading of the manuscript and for giving a number of important comments and suggestions that have greatly improved the presentation of the paper.
This work is supported by the National Natural Science Foundation of China (NSFC) (Nos. 11631004, 71532001), Natural Sciences and Engineering Research Council of Canada (RGPIN-2021-04100).

\end{document}